\documentclass[11pt]{amsart}
\usepackage{amssymb,amsfonts,amsthm,amsmath}
\usepackage{graphicx}
\usepackage{amssymb}
\usepackage[pagewise]{lineno}
\usepackage{mathrsfs}
\usepackage[colorlinks=true, pdfstartview=FitV, linkcolor=blue, citecolor=blue, urlcolor=blue]{hyperref}
\usepackage{booktabs}
\usepackage{tikz}
\usepackage{subfigure}
\usepackage{multicol}
\usepackage{cases}
\usepackage{enumitem}
\usepackage{hyperref}
\usepackage{arydshln} 
\usepackage{float}
\usepackage{color}
\usepackage{mathtools}

\newtheorem{theorem}{Theorem}[section]
\newtheorem{lemma}[theorem]{Lemma}
\newtheorem{proposition}[theorem]{Proposition}
\newtheorem{corollary}[theorem]{Corollary}
\newtheorem{definition}[theorem]{Definition}
\newtheorem{remark}{Remark}

\begin{document}

\title{Modeling and dynamics \\ near irregular elongated asteroids}
	\author{ E.~Mart\'{i}nez\textsuperscript{1}}
	\address{\textsuperscript{1}Departamento de Matem\'atica, Facultad de Ciencias\\Universidad del B\'io-B\'io, Casilla 5-C, Concepci\'on, Chile}
	\author{J.~Vidarte\textsuperscript{2}}
	\address{
		\textsuperscript{2}Departamento de Matem\'atica y F\'isica Aplicadas\\
Universidad Cat\'olica de la Sant\'isima Concepci\'on, Casilla 297, Concepci\'on, Chile}
	\author{J.L.~Zapata\textsuperscript{3}}
	\address{
		\textsuperscript{3}Programa de Formaci\'on Pedag\'ogica para Licenciados y/o Profesionales\\
 Facultad de Educaci\'on, Universidad San Sebasti\'an, Lientur 1457, Concepci\'on, Chile.}
	\date{}

	\begin{abstract}
		
We investigate the qualitative characteristics of a test particle attracted to an irregular elongated body, modeled as a non-homogeneous straight segment with a variable linear density. By deriving the potential function in closed form, we formulate the Hamiltonian equations of motion for this system. Our analysis reveals a family of periodic circular orbits parameterized by angular momentum. Additionally, we utilize the axial symmetry resulting from rotations around the segment's axis to consider the corresponding reduced system. This approach identifies several reduced-periodic orbits by analyzing appropriate Poincar\'e sections. These periodic orbits are then reconstructed into quasi-periodic orbits within the full dynamical system.
	\end{abstract}
	
	\maketitle
\section{Introduction}

Understanding periodic orbits in non-uniform gravitational fields is essential for grasping the dynamic behaviors around asteroids, as well as for the engineering considerations involved in deep space exploration. The complexity arises from the diverse mass distributions of asteroids, which necessitates various models to describe their irregular gravitational fields. Common approaches include polyhedral models \cite{WernerScheeres1996}, mascon models \cite{GEISSLER1996140}, and simplified representations, such as straight segments for elongated bodies \cite{Duboshin1959, Prieto1994, Riaguas1999, Najid2011}.

To our knowledge, the potential of a non-homogeneous straight segment with variable linear density has not been thoroughly explored. While the potential for a homogeneous segment has been detailed in classical works such as MacMillan \cite{MacMillan1958} and Kellogg \cite{Kellogg1967}, with closed-form solutions found in \cite{Duboshin1959, Riaguas1999}. The present paper fills the gap by deriving the closed-form potential for variable density. It also provides a Hamiltonian formulation and examines the existence of periodic and quasi-periodic orbits, offering an efficient model for elongated bodies with asymmetric mass distributions, as an alternative to polyhedral or mascon models.

The quadratic density model analyzed in \cite{Najid2011, Najid2012} assumes a symmetric mass distribution about the center of mass. Consequently, this model is limited to representing bodies with symmetric mass distributions. In contrast, asymmetric mass distributions cannot be captured by such a model. This limitation leaves the dipole model \cite{2018AJ} and our proposed linear density model as the only viable options for accurately modeling elongated bodies with asymmetric mass distributions, making them more suitable for a wider range of applications.

While various models have been proposed to represent the gravitational field of elongated bodies, comprehensive comparisons between them remain scarce. Notably, only one study \cite{2018AJ} has directly compared the gravitational potentials of a straight segment with constant density and the dipole model, highlighting the strengths and limitations of each approach. However, further research is required to compare additional models, such as those involving quadratic or linear densities, in terms of accuracy and computational efficiency. Such comparisons are crucial for advancing our understanding of particle dynamics near irregularly shaped celestial bodies, particularly in the context of asteroid exploration missions. We believe that this task requires a dedicated research, which will be addressed in a future work.

This paper is organized as follows. In Section~\ref{Potential}, we introduce the mathematical model of a non-homogeneous straight segment with variable linear density, deriving the closed-form expression for the gravitational potential. Section \ref{Formulation-Newtonian-Hamiltonian} provides the Hamiltonian formulation of the system. Section~\ref{sec:PeriodicOrbits} and Section~\ref{sec:QuasiPeriodicOrbits} are dedicated to the study of periodic orbits, including circular orbits, and the reconstruction of quasi-periodic orbits from the reduced system.

\section{Potential of a straight segment with linear density}
\label{Potential}
We fix an inertial reference frame $\{O;x,y,z\}$ with the $2L$-long segment $\mathcal S$ lying along the $x$-axis and centered at the origin as in the left image of Figure~\ref{fig:SegmentoFijoPosicionCentrado}. In this frame, we assume that the density is given by the linear function $\sigma(x)=\alpha  x + \beta,\, x\in[-L,L]$. Additionally, our analysis will hinge on the following distances
$$\Delta= |\vec{ \mathcal{P} \mathcal{Q}}|,\quad  r_{1}=  |\vec{\mathcal{P} E_1}|,\quad  r_{2}= |\vec{\mathcal{P}E_2}|,$$
where $ \mathcal{P}=(x,y,z)$ denotes a massless particle, $ \mathcal{Q}\in\mathcal S$ is arbitrary, and $E_1= (L,0,0)$ and $E_2=(-L,0,0)$ are the end points of the segment. 
Keeping in mind that the total mass $M$ of the segment and the density are always positive, we obtain the following restrictions on the parameters. 
\begin{proposition}
The following statements hold:
\begin{enumerate}[label=\roman*)]
  \item The total mass $M$ of the segment is given by $M=2L\beta$.
  \item The coordinates of the center of mass are $CM=\left(\bar{c},0,0\right)$, where $\bar{c}=\frac{2 \alpha L^3}{3 M}$. 
  \item The slope of the density function satisfaces:  $- \frac{M}{2L^2}< \alpha < \frac{M}{2L^2}.$
\end{enumerate}
\end{proposition}
\begin{proof}
The total mass $M$ of the segment and the $x$-coordinate $\bar{c}$ of the center of mass are given respectively by
$$ M =  \int_{-L}^L \sigma(x)\, dx,\quad \bar{c}= \frac{1}{M}\int_{-L}^L x\,\sigma(x) dx,$$
which, after imposing $\sigma(x)=\alpha  x + \beta$, lead to { (i)} and { (ii)}. Item { (iii)} follows from solving $\sigma(x)=\alpha  x + \beta>0,$ with $\beta=M/2L$ and $ x\in [-L,L].$
\end{proof}

\begin{figure}[h]
\vspace{-1.0cm}\includegraphics[scale=0.46]{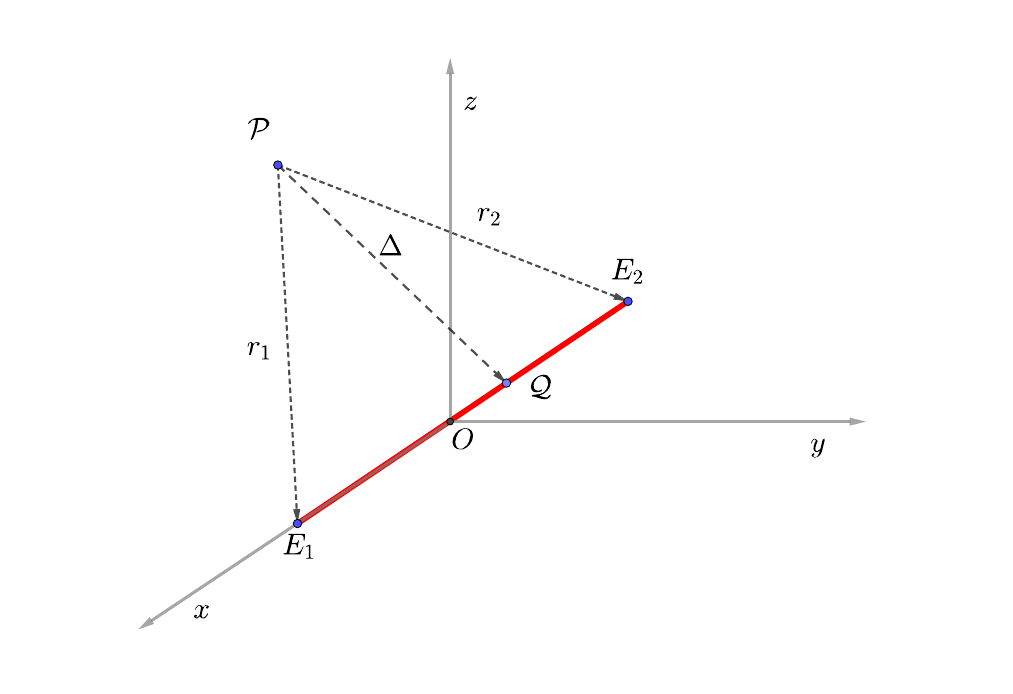}
\hspace{-1.5cm}\includegraphics[scale=0.5]{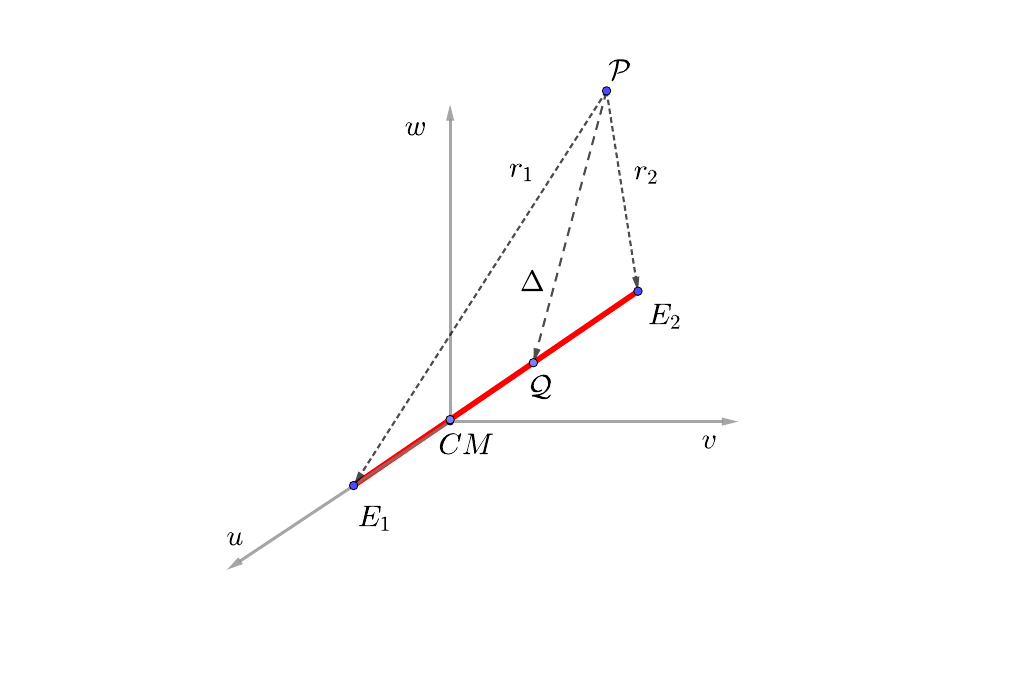}
\vspace{-0.5cm}\caption{\footnotesize The fixed segment in different reference frames. Left: The reference frame $\{O;x,y,z\}$ with the segment centered on the origin and lying along the $x$-axis. Right: The new frame $\{O;u,v,w\}$ results from a translation in the $x$-axis locating the center of mass at the new origin.}
\label{fig:SegmentoFijoPosicionCentrado}
\end{figure}

\begin{remark}
\label{remark:Slope}
In what follows, we will restrict our study to positive values of the parameter $\alpha$, since the negative values are obtained just by switching the orientation of the $u$-axis.
\end{remark}

The computation of the potential function is done in a reference frame $\{O;u,v,w\}$ where the segment lies along the $u$-axis with its center of mass located at the origin, see the right image of Figure~\ref{fig:SegmentoFijoPosicionCentrado}. In this frame, the distances from $\mathcal{P}$ to the end points are
\begin{equation}
 r_{1}= \sqrt{(-L +u+\bar{c})^2+v^2 + w^2},\quad  r_{2}= \sqrt{(L + u+ \bar{c})^2+v^2 + w^2},
\end{equation}
and the density function is $\sigma(u) =\alpha u +\beta+\alpha \bar{c}$, with $ u\in[-L-\bar{c},L-\bar{c}].$ Hence, the potential function is given by the following line integral
  \begin{equation}
 \begin{array}{rl}
V(\mathcal{P}) & =  -G \displaystyle\int_\mathcal{S} \dfrac{d{\textrm M}}{\Delta}= -G \displaystyle\int_{-L-\bar c}^{L-\bar c} \dfrac{\sigma(u)\ du}{|\vec{ \mathcal{P} \mathcal{Q}}|}. \end{array}
\end{equation}
To simplify the denominator, we express $\vec{ \mathcal{P} \mathcal{Q}}= r_1 + s\,r_{12}$, where $r_{12}=r_2 -r_1=2L$ with $s\in(0,1)$, which leads us to $|\vec{ \mathcal{P} \mathcal{Q}}|=\sqrt{r_1^2 + s^2\,r_{12}^2+2s\, r_1 r_{12}}$.  The above integral is computed, in terms of the new variable $s$, by considering the change of variable $u=2Ls-L- \bar{c}$, $du=2Lds$. Thus, considering the relation $2 r_1 r_{12}=r_2^2 -r_1^2 -4L$ obtained from the Cosine Theorem, we are left with the following integral
\begin{equation}
 \begin{array}{rl}
V(\mathcal P) 
 & = -G \displaystyle\int_0^1 \frac{c_1s +c_2 }{\sqrt{s^2+c_3 s+c_4}} \ ds,
\end{array}
\end{equation}
where $c_1= 2L\alpha$, $c_2= \beta-\alpha L $, $c_3= (-4 L^2-{r}_{1}^2+{r}_{2}^2)/4 L^2$ and $c_4= {r}_{1}^2/4 L^2$. The final expression for the potential function is 
\begin{equation}
\label{eq:PotencialOriginal3}
V(u, v, w)=\frac{ c_1 G}{2 L} ( r_{1}- r_{2})-\frac{ G}{8 L^2}  \left[ c_{1} \left(4 L^2+ r_{1}^2- r_{2}^2\right)+ 8c_{2}   L^2\right] \ln \left(\frac{2 L+{r_1}+{r_2}}{-2 L+{r_1}+{r_2}}\right).
\end{equation}
Notice, that $c_1=0$ leads us to the potential associated with constant density, see \cite{Kellogg1967,Duboshin1959, Riaguas1999}. Moreover, we highlight that $V$ only depends on the distances $  r_1$ and $ \mathbf r_2$. Therefore, the potential function remains the same under rigid transformations of the reference frame.

\section{Equations of motion. Hamiltonian formalism}
\label{Formulation-Newtonian-Hamiltonian}
For our subsequent analysis, we introduce the conjugate momenta $p=(p_u,p_v,p_w)$ associated with $q=(u,v,w)$. Then, the equations of motion can be written as a Hamiltonian system
\begin{equation}
\dot{q}=H_p,\quad \dot{p}=-H_q,\nonumber
 \end{equation}
where $H(q,p;\alpha, M, L)$ is a multiparametric family of Hamiltonian functions given by
\begin{gather}
\begin{aligned}
\label{hamiltoniano-rotando1}
H(q,p)=&\frac{1}{2} \left(p_u^2+p_v^2+p_w^2\right)
+V(q),
 \end{aligned}
 \end{gather}
 with $V(q)$ the potential function given in \eqref{eq:PotencialOriginal3}. The domain of the Hamiltonian $H$ is $ \mathbb{R}^3 \setminus  \mathcal{S}\times\mathbb{R}^3$, being $\mathcal{S}= \{q\in \mathbb{R}^3 : u \in [-L-\bar{c}, L-\bar{c}], \ v=0, \ w=0  \}$.

To reduce the number of parameters, we customize the units through the following symplectic change of scale $(q,p)\to(Q,P)$, and time reparametrization
\begin{equation}
\label{scaling}
q\to L Q,\quad p\to \sqrt{\frac{G M}{2 L}} P, \quad \alpha \to A (3 M)/2 L^2,\quad t=\frac{2 L}{G M}\tau,
\end{equation}
where $Q=(\xi,\eta,\zeta)$, $P=(p_\xi,p_\eta,p_\zeta)$. This transformation provides the new family of Hamiltonians depending on one parameter, $\mathcal H(Q,P)=\mathcal H(Q,P;A)$, given by
\begin{gather}
\begin{aligned}
\label{eq:HamiltonianaEscalada}
\mathcal H(Q,P)=& \frac{1}{2} \left(p_\xi^2+p_\eta^2+p_\zeta^2\right)
+U(Q;A),
 \end{aligned}
 \end{gather}
with  $0\leqslant A\leqslant{1}/{3}$. The domain of $\mathcal H$ is $\mathbb{R}^3 \setminus  \mathcal{S}\times\mathbb{R}^3 $, where  $\mathcal{S}= \{Q \in \mathbb{R}^3 : \xi \in [-1-A, 1-A], \ \eta=0, \ \zeta=0  \}$. Moreover, the potential is obtained through the following compact expression
\begin{equation}
\label{eq:PotencialEscalado}
U(Q;A)=3 A \,d-\frac{1}{4}(3 A \,d\, s+4) \ln \left(\frac{s+2}{s-2}\right),
\end{equation}
where the auxiliary variables $s$ and $d$ are given by

\begin{equation}
\label{eq:VariablesAuxiliares}
s=R_1+R_2,\quad d=R_1-R_2, 
\end{equation}
with $R_1=\sqrt{( -1+A+\xi)^2+\eta ^2+\zeta ^2}$, and $R_2=\sqrt{(1+A+\xi )^2+\eta ^2+\zeta ^2}.$ 

The use of the auxiliary variables $s$ and $d$ will facilitate the analytic study by allowing much more compact expressions. Furthermore, these variables satisfy the following properties.

\begin{lemma}
\label{lemma:sd}
The variables $s$ and $d$ satisfy the following properties:
\begin{enumerate}[label=\roman*)]
  \item $2< s$ and $-2\leqslant d\leqslant2$.
  \item $s=2$, the  infinitesimal particle $\mathcal{P}$ collide to the segment $\mathcal S$.
  \item $ d=-2$, if and only if, $\eta=\zeta=0$ and $\mathcal{P}$ is on the right side of the segment. 
  \item $ d=2$, if and only if, $\eta=\zeta=0$ and $\mathcal{P}$ is on the left side of the segment.
\end{enumerate}
\end{lemma}
\begin{proof}
After the scaling, the longitude of the segment is equal to two. Hence, the particle $\mathcal{P}$ and end points $E_1$ and $E_2$ generate a triangle with sides equal to $R_1$, $R_2$, and 2. According to triangle inequality and the reverse triangle inequality we have
$$s=R_1+R_2>2,\quad |d|= |R_1-R_2|< 2 .$$
 Hence, the properties of the lemma hold.
 \end{proof}

After all these manipulations, we may give the equations of motion in the following compact form
 \begin{gather}
\begin{aligned}
\label{eq:EqMovSegmentoFijo}
 \dot{\xi}&= p_\xi ,\qquad &\dot{p_\xi}&= \frac{4\,(4d+12 A\, s)}{s^2-d^2}-3 A \ln \left(\frac{s+2}{s-2}\right), \\ 
\dot{\eta}&= p_\eta,\qquad &  \dot{p}_\eta&=16\,\eta\,\frac{3 A\, d+s}{\left(s^2-4\right) \left(d^2-s^2\right)}, \\ 
\dot{\zeta}&= p_\zeta,\qquad & \dot{p}_\zeta&=-16\,\zeta\,\frac{3 A\, d+s}{\left(s^2-4\right) \left(s^2-d^2\right)}.
 \end{aligned}
 \end{gather}
 
Notice that the above system \eqref{eq:EqMovSegmentoFijo} is invariant under rotations about $\xi$-axis. Moreover, all the planes that contain the segment are invariant. In particular, the $\xi$-axis, given by $\eta=\zeta=0$ and $p_\eta=p_\zeta=0$, is invariant.

Due to the axial symmetry of the problem, it is convenient to define the following symplectic transformation $(\xi,\eta, \zeta, p_\xi,p_\eta, p_\zeta) \rightarrow (r, \theta , x, P_r , P_\theta, P_x)$
\begin{gather}
\begin{aligned}
\label{hamiltoniano-fijo-cilindricas}
\xi=&\,x,\quad &p_\xi=&P_x,\\ 
\eta=&\,r \cos\theta, \quad &p_\eta=&(r P_r \sin_\theta + P_{\theta} \cos\theta)/r,\\ 
 \zeta=&\,r \sin\theta, \quad &p_\zeta=&(r P_r \cos\theta- P_{\theta}\sin\theta)/r.
 \end{aligned}
 \end{gather}
Thus, the Hamiltonian function in cylindrical coordinates is expressed as 
$$\mathcal H=\dfrac{1}{2} \left(P_r^2 + P_x^2+ \dfrac{ P_\theta^2}{r^2}\right)+ 
 3 A \,d-\frac{1}{4}(3 A \,d\, s+4) \ln \left(\frac{s+2}{s-2}\right),$$
and the equations of motion are the following
\begin{gather}
\begin{aligned}
\label{sistema-ham-segmento-fijo-densidad-lineal1}
 \dot{r}&= P_r,\qquad &\dot{P_r}&=\frac{P_\theta^2}{r^3}-\frac{16r(3 A d+s)}{(s^2-d^2)(4-s^2) }, \\ 
\dot{\theta}&= P_\theta/r^2,\qquad & \dot{P_\theta}&=0, \\ 
\dot{x}&= P_x,\qquad & \dot{P_x}&=  \frac{4\,(4d+12A\, s)}{s^2-d^2}-3 A \ln \left(\frac{s+2}{s-2}\right).
 \end{aligned}
 \end{gather}

  \begin{proposition}
  \label{propo:NoEq}
 The Hamiltonian system \eqref{sistema-ham-segmento-fijo-densidad-lineal1} has no equilibria.
 \end{proposition}
 \begin{proof}
 An equilibrium implies $P_r=P_\theta=P_x=0$. Thus, the equation for $\dot{P_r}$ becomes
 $$\frac{16r(3 A d+s)}{(s^2-d^2)(4-s^2) }=0,$$
whose numerator is always strictly positive.
 \end{proof}
 
 \section{Periodic orbits of the non-homogeneous fixed straight segment}
 \label{sec:PeriodicOrbits}
Since the polar angle is cyclic in system \eqref{sistema-ham-segmento-fijo-densidad-lineal1}, we investigate the existence of periodic orbits on the reduced system obtained after disregarding the variables $(\theta,P_\theta)$. Note that the presence of $P_\theta$ in the following equations of motion plays the role of an arbitrary constant parameter, which from now on we denote $P_\theta=c$
 \begin{gather}
\begin{aligned}
\label{sistema-ham-segmento-fijo-densidad-lineal1-reducido}
 \dot{r}&= P_r,\qquad &\dot{P_r}&=\frac{c^2}{r^3}-\frac{16r(3 A d+s)}{(s^2-d^2)(4-s^2) }, \\ 
\dot{x}&= P_x,\qquad &\dot{P_x}&=  \frac{4\,(4d+12A\, s)}{s^2-d^2}-3 A \ln \left(\frac{s+2}{s-2}\right).
 \end{aligned}
 \end{gather}
 The inverse relation between the pairs $(r,x)$ and $(s,d)$ follows by solving the equations for $r$ and $x$, where now the radius in cylindrical coordinates are
 \begin{gather}
\begin{aligned}
\label{variables cilindricas-s-d}
d=&r_1-r_2=\sqrt{(A+x-1)^2+r^2}-\sqrt{(A+x+1)^2+r^2}, \\ 
s=&r_1+r_2=\sqrt{(A+x-1)^2+r^2}+\sqrt{(A+x+1)^2+r^2}. \\ 
 \end{aligned}
 \end{gather}
 Hence, we obtain the following relation that will be useful in the study of the equilibria for the system \eqref{sistema-ham-segmento-fijo-densidad-lineal1-reducido}  
\begin{equation}
\label{relacion r-x-s-d}
r(s,d)=\frac{1}{4} \sqrt{\left(s^2-4\right)\left(4-d^2\right) },\quad x(s,d)=-A-\frac{d s}{4}.
\end{equation}

After making some arrangements in \eqref{sistema-ham-segmento-fijo-densidad-lineal1-reducido} comes down to study the roots of the following two functions in $s$ and $d$
\begin{gather}
\begin{aligned}
\label{sistema-ham-reducido-segmento-densidad-lineal3}
F_1&=\left(d^2-4\right)^2 \left(s^2-4\right) (3 A d+s)+16 c^2 \left(d^2-s^2\right), \\ 
F_2 &=3 A s+d+\frac{3}{4} A \left(d^2-s^2\right) \ln \left(\frac{s+2}{s-2}\right). \\ 
 \end{aligned}
 \end{gather}

Since $r=r(s,d)$ and $x=x(s,d)$ from \eqref{relacion r-x-s-d}, if we found a solution $(s^*,d^*)$ in \eqref{sistema-ham-reducido-segmento-densidad-lineal3}, then we obtain a critical point $(r(s^*,d^*),x(s^*,d^*),0,0)$ of the system \eqref{sistema-ham-segmento-fijo-densidad-lineal1-reducido}.

 In \cite{Riaguas1999}, the author proved the existence of a circular orbit in an invariant perpendicular plane to the segment through the origin for the case $A=0$. This orbit is obtained for 
 $$s_0=\frac{1}{2} \left(c^2+\sqrt{c^4+16}\right), \quad d_0=0,\quad P_\theta=c, \quad P_r=P_x=0.$$
 The following result shows that a unique circular solution exists for all values of parameter $A$ small enough. 
 
 \begin{theorem}
 For $A$ small enough, there exist a unique periodic circular orbit
  $$s=s(A)=s_0+\mathcal O(A^2), \quad d=d(A)=d'(0)\,A+\mathcal O(A^2),\quad P_\theta=c,\quad P_r=P_x=0.$$
 \end{theorem}
 \begin{proof} 
 For this purpose, we consider the function $F(s,d; A)=(F_1, F_2)$ and employ the implicit function theorem to prove the existence of a unique differentiable functions $s(A)$ and $d(A)$, satisfying $F(s(A),d(A); A)=0$ in a small neighborhood of $A=0$. First, note that $F(s_0,0,0)=0$ and second, computing

 $$DF(s_0,0;0)=\left(\begin{array}{cc}  \dfrac{\partial F_1}{\partial s} &  \dfrac{\partial F_1}{\partial d} \\ [1.4ex]\dfrac{\partial F_2}{\partial s} &  \dfrac{\partial F_2}{\partial d}\end{array}\right)_{A=0,d=0,s_0}=\left(\begin{array}{cc}  8 \left(c^4+\sqrt{c^4+16} c^2+16\right) & 0 \\ [1.4ex] 0 & 1\end{array}\right).$$
 Thus, since $det|DF(s_0,0;0)|= 8 \left(c^4+\sqrt{c^4+16} c^2+16\right)\neq 0$, by the implicit function theorem, we conclude the proof.
 \end{proof}
 Moreover, we may ensure the existence of the circular orbit for a wider domain of the parameter $A$. More precisely, we regard the equation $F_2$ as a second-degree polynomial in the variable $d$. Thus, assuming $A\neq0$ in $F_2$ and solving in the variable $d$, we obtain 
 
\begin{figure}[h]
\centering
\vspace{0.0cm}
\includegraphics[scale=0.5]{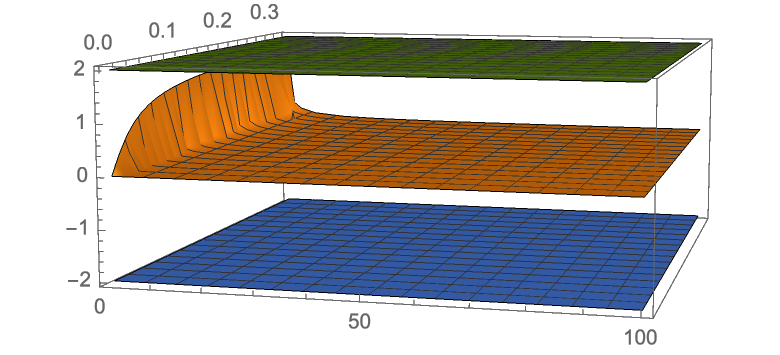}
\includegraphics[scale=0.45]{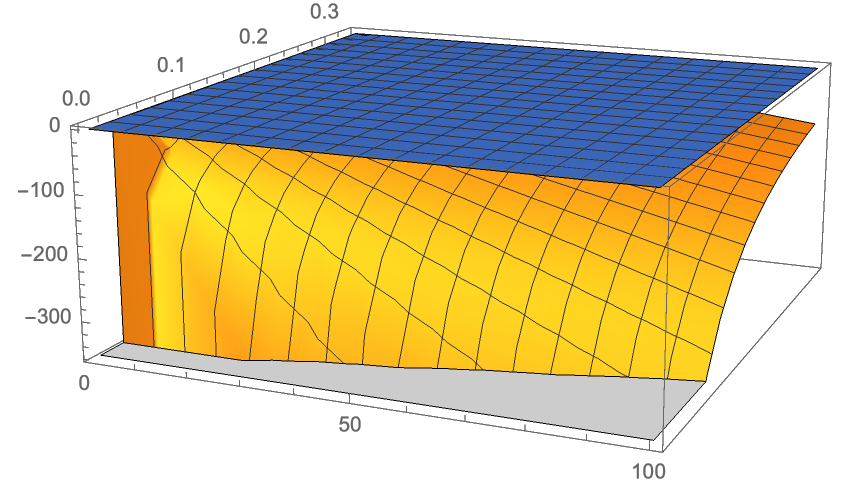}
\caption{\footnotesize  Left and right figures shows the surfaces $d_+(s;A)$ and $d_-(s;A)$, for $s\in(2,100)$ and the parameter $A\in[0,1/3)$. Green and blue planes are $d=2$ and $d=-2$ respectively. In the domain of these figures we observe the brown surface, $d_+(s)\in(0,2)$ and the yellow plane, $d_-(s)\in(-\infty,-2)$.}
\label{fig:dMm}
\end{figure}

 \begin{equation}
\label{valor d-circular}
d_{\pm}(s;A)=\frac{-2\pm\sqrt{\left(3 A s \ln \left(\frac{s+2}{s-2}\right)-6 A\right)^2+4-36 A^2}}{3 A \ln \left(\frac{s+2}{s-2}\right)}.
\end{equation}
Note that, since $A\in[0,1/3)$ and $s\in(2,+\infty)$ the radicand is always a positive number, hence $d_{\pm}(s; A)\in \mathbb{R}$. We denote $d_{\pm}(s;A)=d_{\pm}(s)$ for a cleaner notation. 
Moreover, setting $s=s^*$ and $d=d^*=d_{\pm}(s^*)$ in equation $F_1=0$, and fixing, we arrive to

$$|c^*|=\dfrac{(4-d^{*2})}{2}\sqrt{\frac{ \left(s^{*2}-4\right) (3 A d^*+s^*)}{ s^{*2}-d^{*2}}}.$$
Considering that all the factors inside the square root are positive numbers, we have $c^*\in \mathbb{R}$. For our problem to be physically meaningful, it is required that $-2\leq d_{\pm}(s)\leq 2$. Taking into account the formula \eqref{valor d-circular}, then the limits of $d_\pm(s)$ in the interval $(2,\infty)$ are given by  
$$\underset{s\to 2^+}{\text{lim}}d_-(s)=-2,\quad \underset{s\to \infty}{\text{lim}}d_-(s)=-\infty,\quad \underset{s\to 2^+}{\text{lim}}d_+(s)=2,\quad \underset{s\to \infty}{\text{lim}}d_+(s)=0.$$
Secondly, we have numerical evidence that $d_-(s)\in(-\infty,-2)$ and $d_+(s)\in(0,2)$, see Figure~\ref{fig:dMm}. Therefore, when $A\neq0$ we have that, for each fixed $s^*$, there is only one posible value for $d^*=d_+(s^*)$ satisfying that $F_2(s^*,d(s^*))=0$.

\begin{figure}[H]
\centering
\includegraphics[scale=0.4]{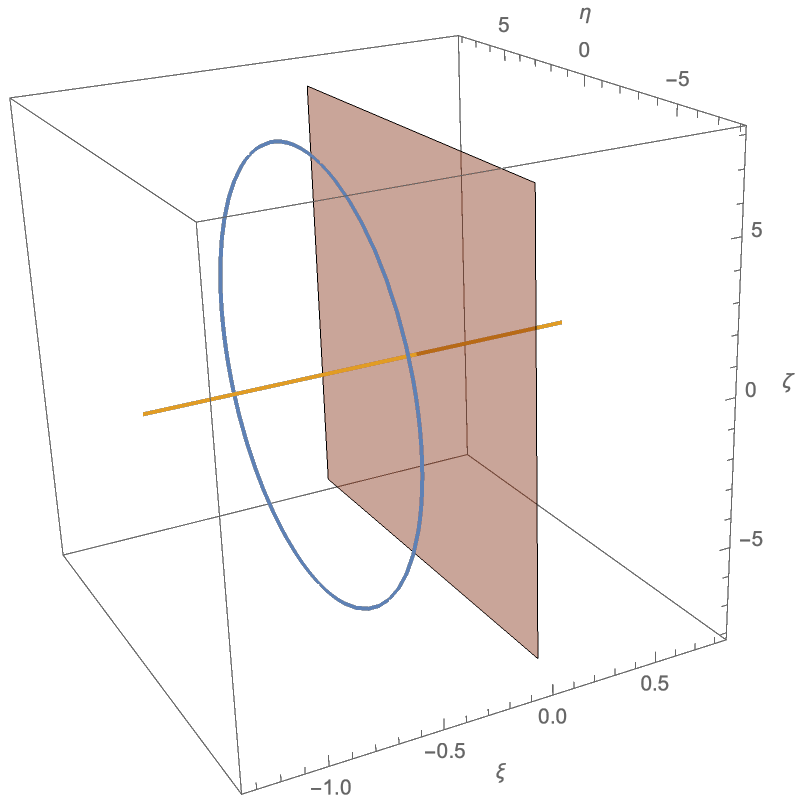}
\caption{\footnotesize  This figure shows the plot of the periodic orbit, The yellow straight line is a representation of the segment. In contrast with the constant density case $A=0$, the periodic orbit does not lie on the $\eta\zeta$-plane. For $A=0.25$, the initial conditions of this orbit are $r=4.8926$, $\theta=0$, $x=-0.5042$, $P_r=0$, $P_\theta=3.1023$, $P_x=0$.}
\label{fig:dMm2}
\end{figure}


\begin{theorem}[Circular orbit]
\label{circular-orbit}
For each fixed $s=s^*$, we consider $(s^*,d^*,c^*)$ as they were defined above. Then, the system \eqref{sistema-ham-segmento-fijo-densidad-lineal1} has a unique circular orbit with period $T=2\pi r^{*2}/c^*$, where is obtained as $r^*=r(s^*,d^*)$ by using the relations \eqref{relacion r-x-s-d}. 
 \end{theorem}
\begin{proof}
Assuming the hypothesis $d=d^*\in(0,2)$, we have that the equations \eqref{sistema-ham-reducido-segmento-densidad-lineal3} vanish. Moreover, considering $P_r=P_x=0$ and the two possible values of $c^*$, (prograde and retrograde), hence we obtain the following pair of solutions of system \eqref{sistema-ham-segmento-fijo-densidad-lineal1} 
  
$$r(t)=r^*,\quad \theta(t)=\dfrac{c^*}{r^{*2}}t+\theta_0,\quad x(t)=x^*,\quad P_r(t)=P_x(t)=0,\quad P_\theta(t)=c^*,$$
where $ x^*=x(s^*,d^*)$. Hence, an equilibrium point $(r^*,x^*, 0,0)$ give us a circular solution of the form 
\begin{gather}
\begin{aligned}
\label{solucion-fijo-cilindricas}
\xi=&\,x^*,\quad \eta=&\,r^* \cos\left({c^*}/{r^{*2}}t+\theta_0\right),\quad \zeta=&\,r^* \sin\left({c^*}/{r^{*2}}t+\theta_0\right).
 \end{aligned}
 \end{gather}
Moreover, considering the above expression for $\theta(t)$, we have $T=2\pi r^{*2}/c^*$.
\end{proof}

\begin{remark}
The circular orbit is in the $\eta\zeta$-plane only for the constant density case $A=0$ \cite{Riaguas1999}. A Taylor expansion for the function $d(A)$ in a small neighborhood of $A$ shows that the orbit shifts to the left as it is seen in Figure~\ref{fig:dMm2}
\begin{equation}
\label{desplazamiento-orbita }
d(A)=d(0) + d'(0)A + \mathcal O (A^2)= {3x}/{16} \left[x \ln \left({(x+4)}/{(x-4)}\right)-8\right]A + \mathcal O (A^2),
\end{equation}
with $x=\sqrt{c^4+16}+c^2$. Thus, $d'(0)$ is a function in $c^2\in(0,+\infty)$. However, since $x\in(4,+\infty)$ and $c^2\in(0,+\infty)$ are in one to one correspondence, we study the sign of $d'(0)$ using the variable $x$. Note that the expression $d'(0)\to0$ as $x\to\infty$, and it is strictly decreasing in $x$. Hence, $d'(0)>0$ for all $x\in(4,+\infty)$ and equivalently for all $c^2\in(0,+\infty)$.
\end{remark}

\begin{figure}[H]
\includegraphics[scale=0.24]{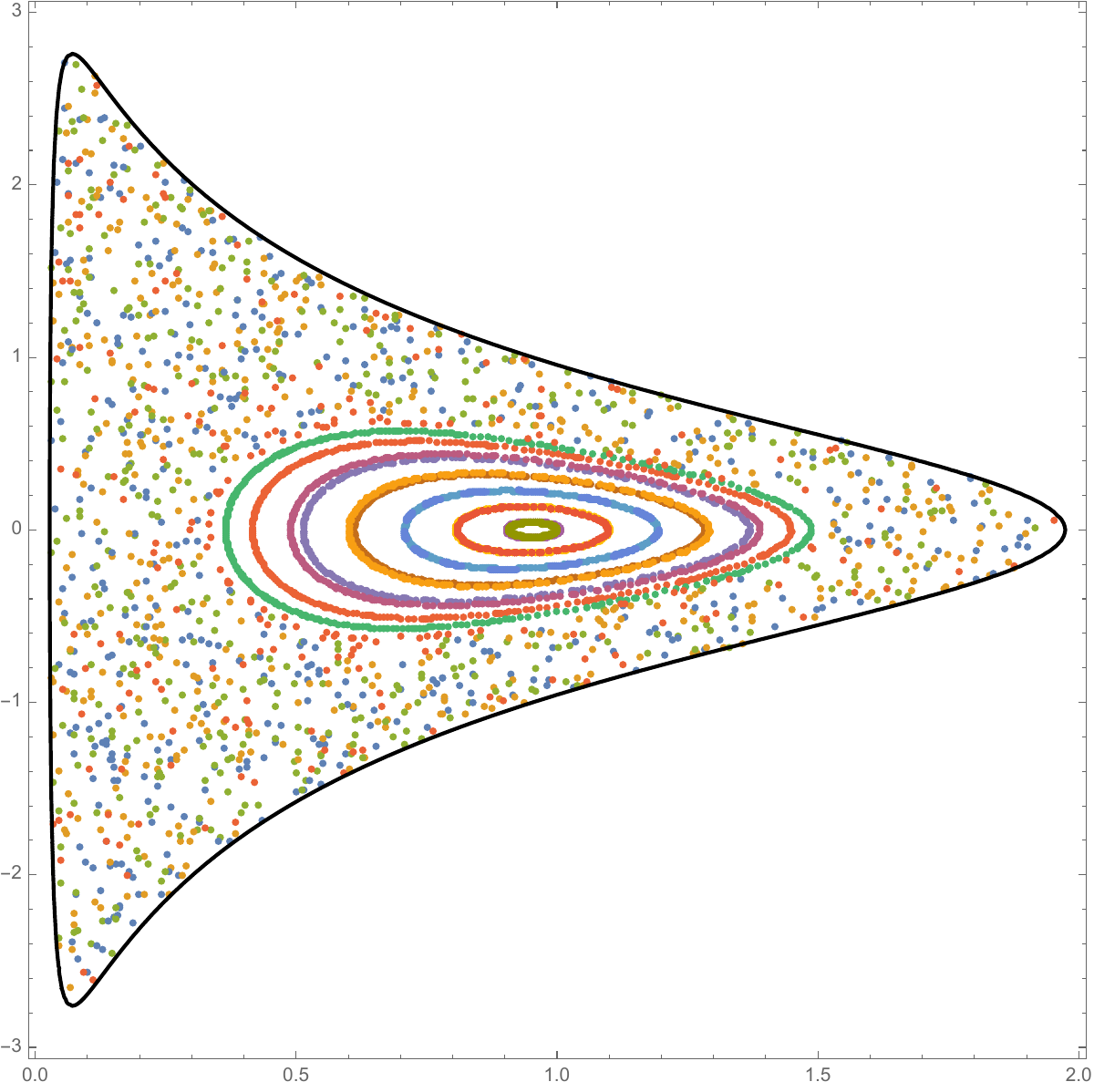}
\includegraphics[scale=0.24]{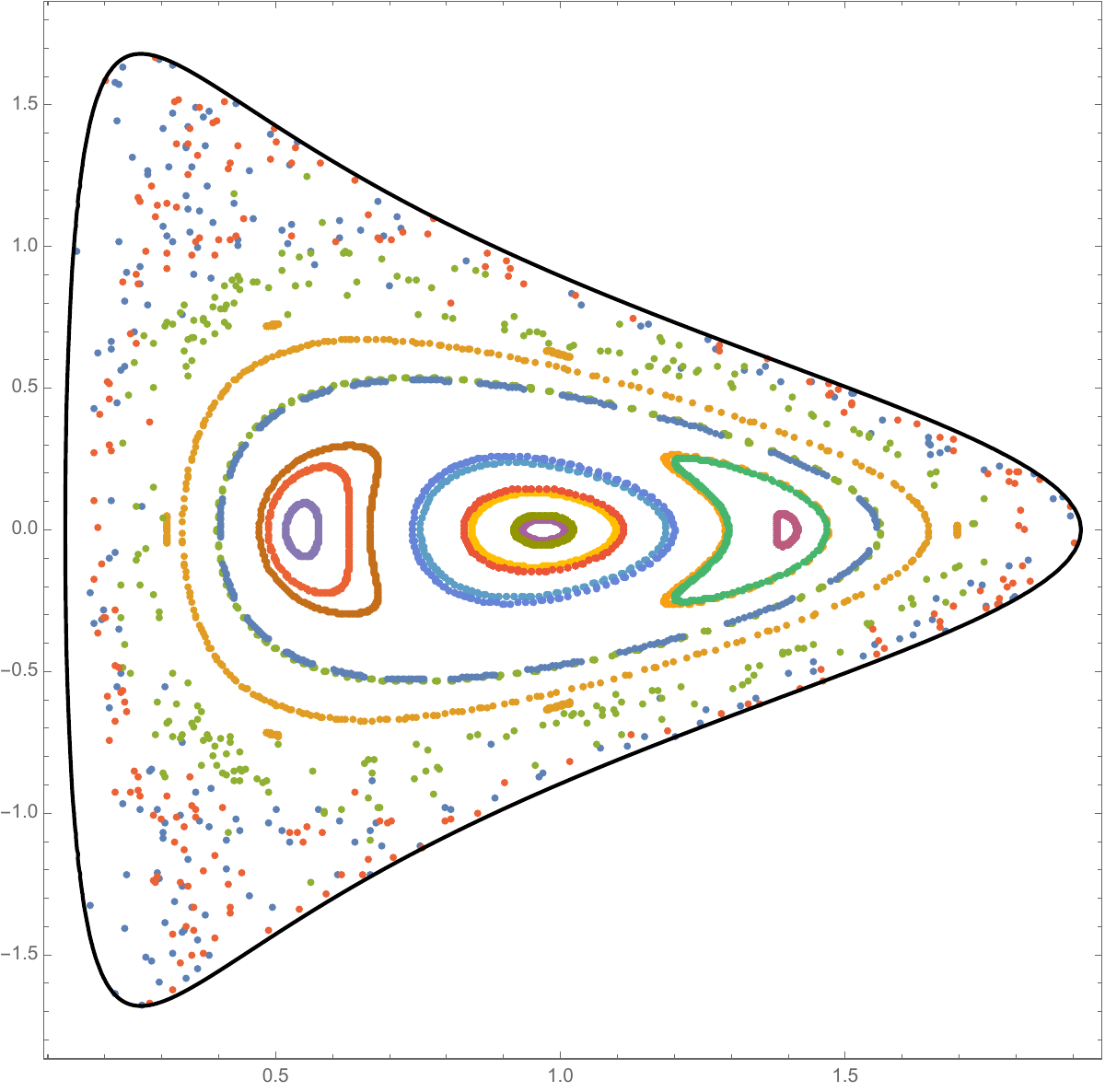}
\includegraphics[scale=0.24]{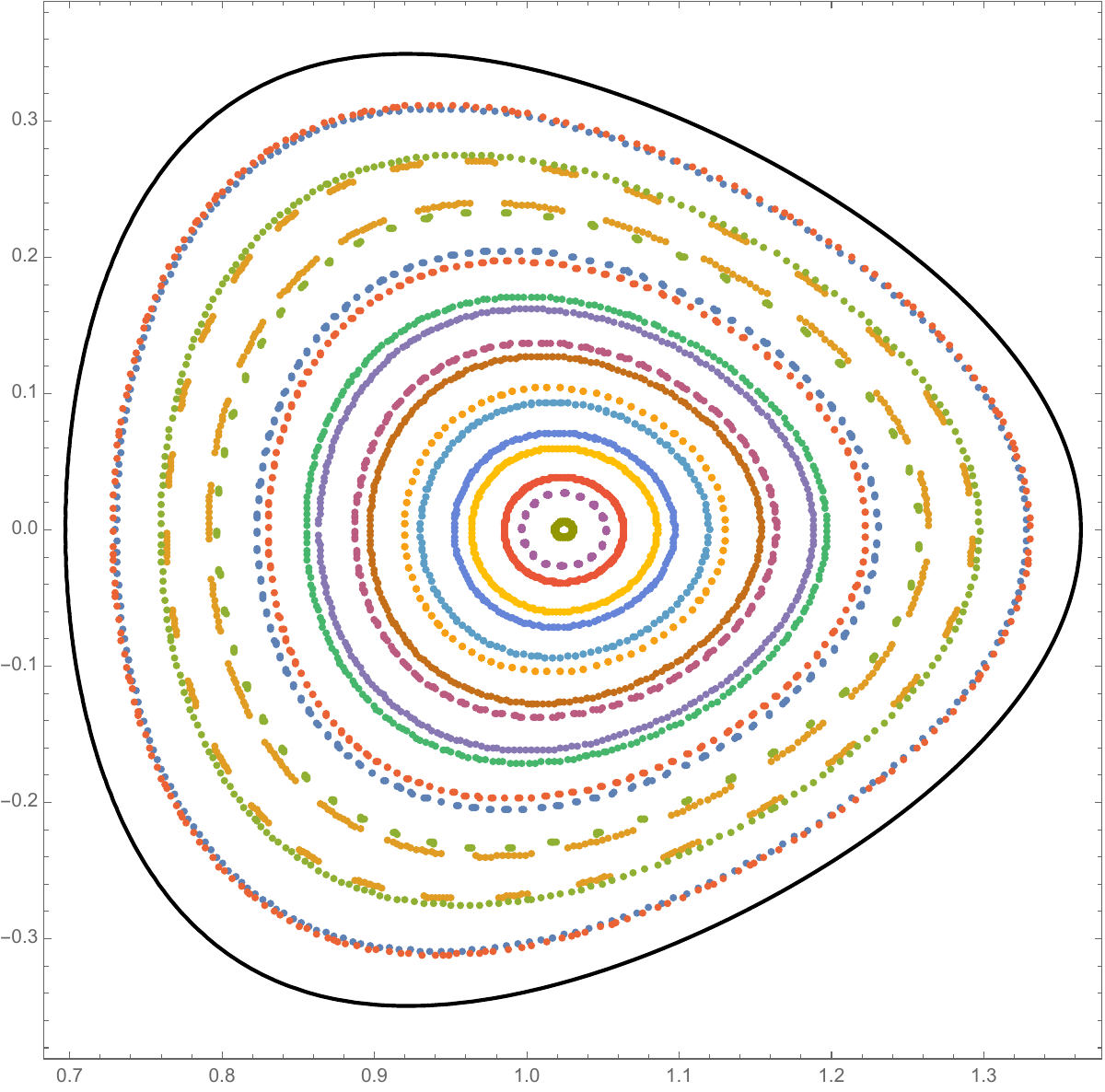}
\caption{\footnotesize Poincar\'e sections for $A=0$: From left to right, $c= 0.1,\, 0.35,\, 0.9$. These sections reproduce the analysis from \cite{RiaguasElipeLara1999} and are constructed using the same units as in that study. Specifically, the length unit is taken as $2L$.}
\label{fig:SeccionesPoincareA0}
\end{figure}

\section{Quasi-periodic orbits around a non-homogeneous fixed straight segment}
\label{sec:QuasiPeriodicOrbits}
In the previous section, we obtained periodic orbits with constant $r(t)=r_0$ and $x(t)=x_0$. In this section, we will search for periodic orbits in the reduced space $(r,x,P_r,P_x)$. The 3-dimensional reconstruction of these reduced-periodic orbits leads to quasi-periodic orbits of the full system in $(r,\theta,x,P_r,P_\theta,P_x)$, unless any commensurability condition is imposed among the momenta.
\begin{figure}[h]
\includegraphics[scale=0.18]{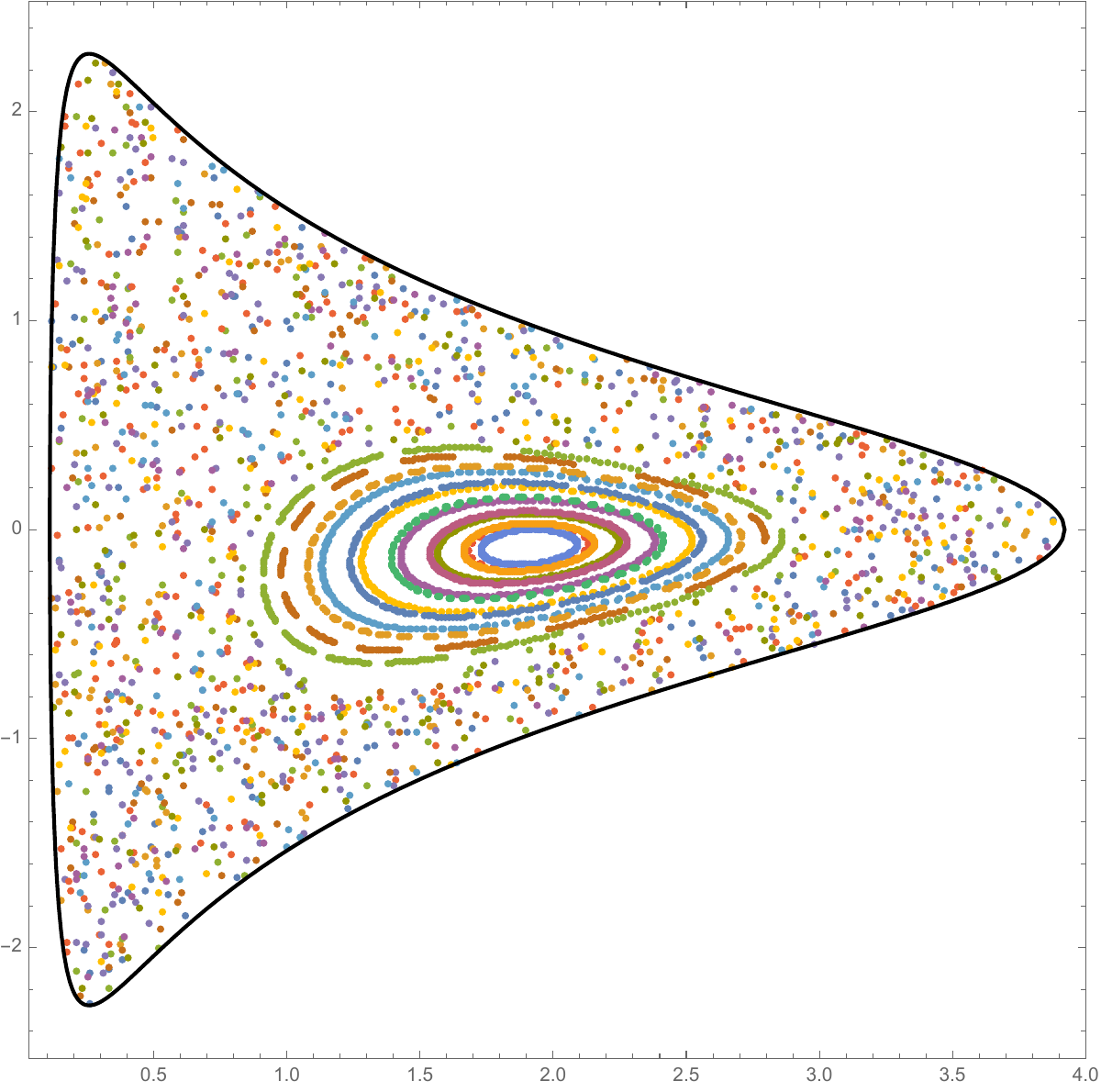}
\includegraphics[scale=0.18]{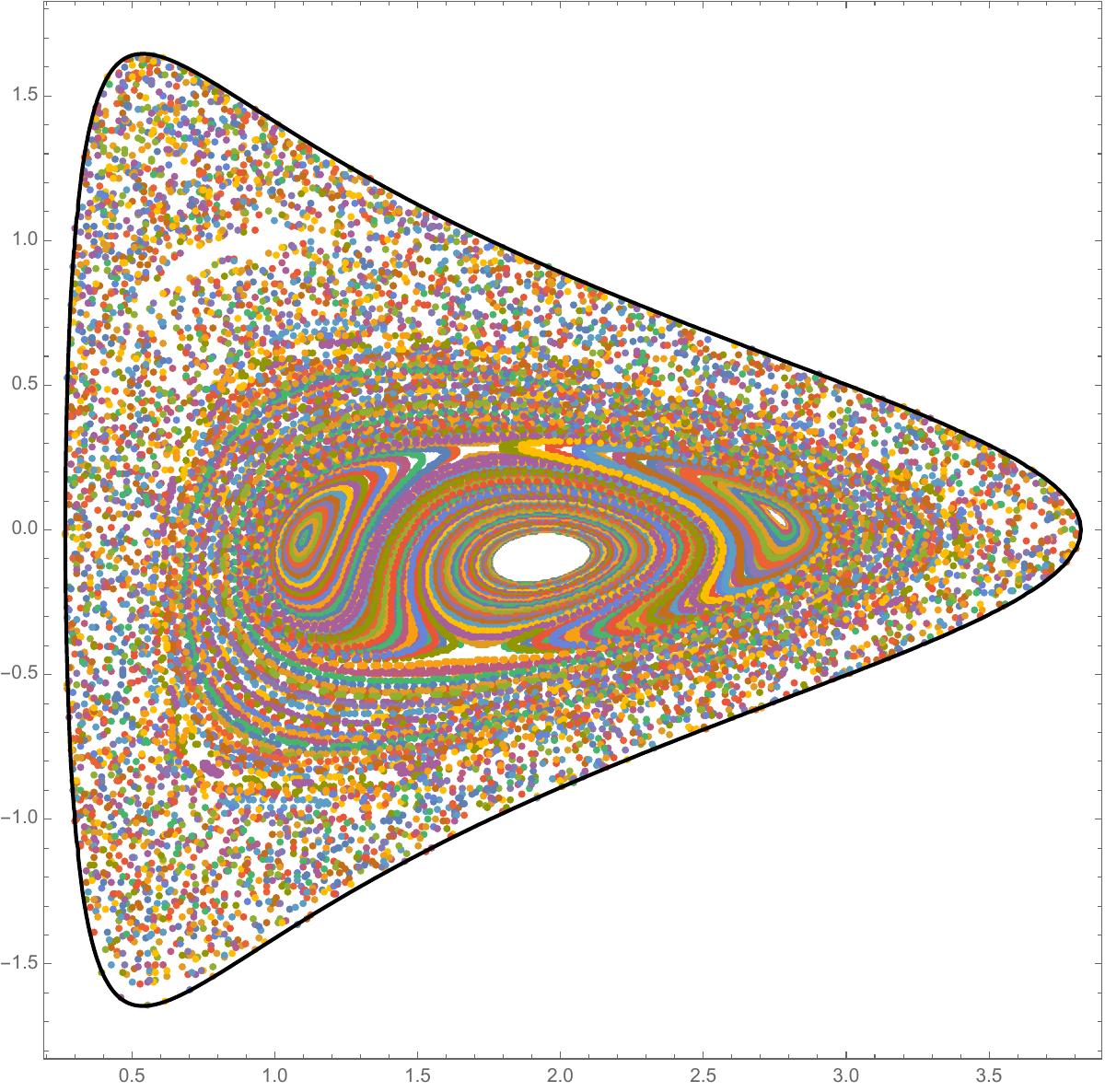}
\includegraphics[scale=0.18]{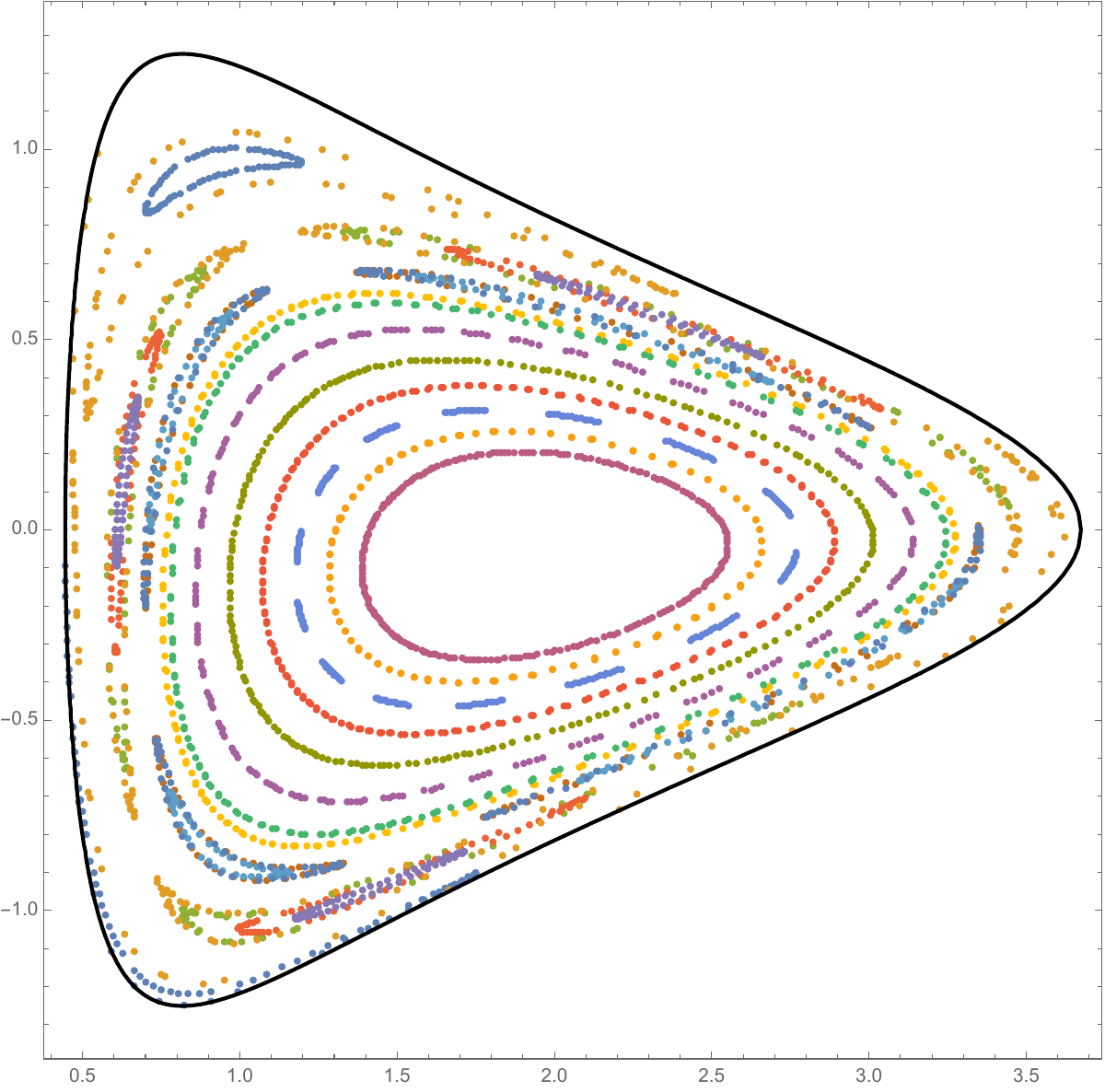}
\includegraphics[scale=0.18]{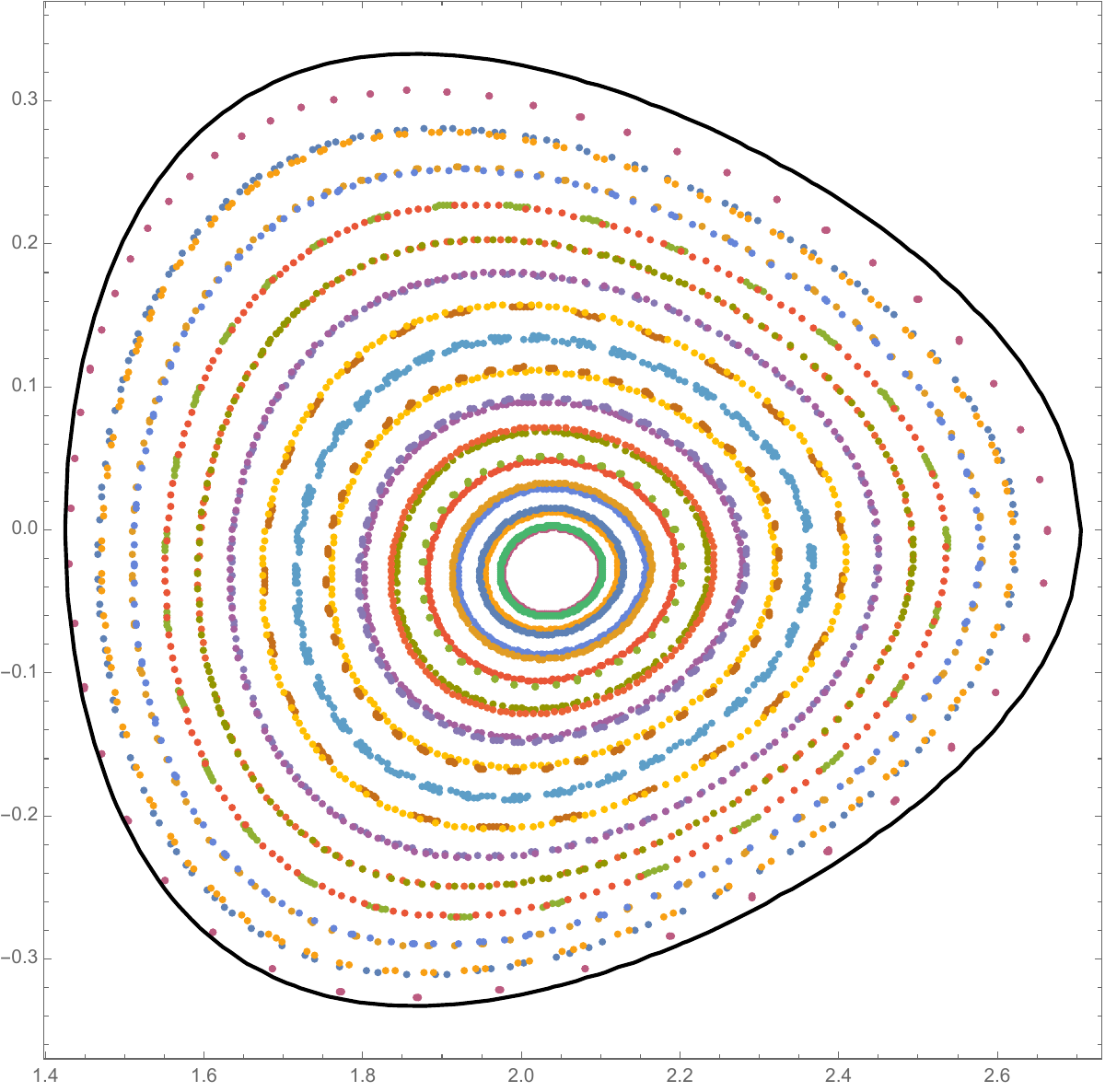}
\includegraphics[scale=0.18]{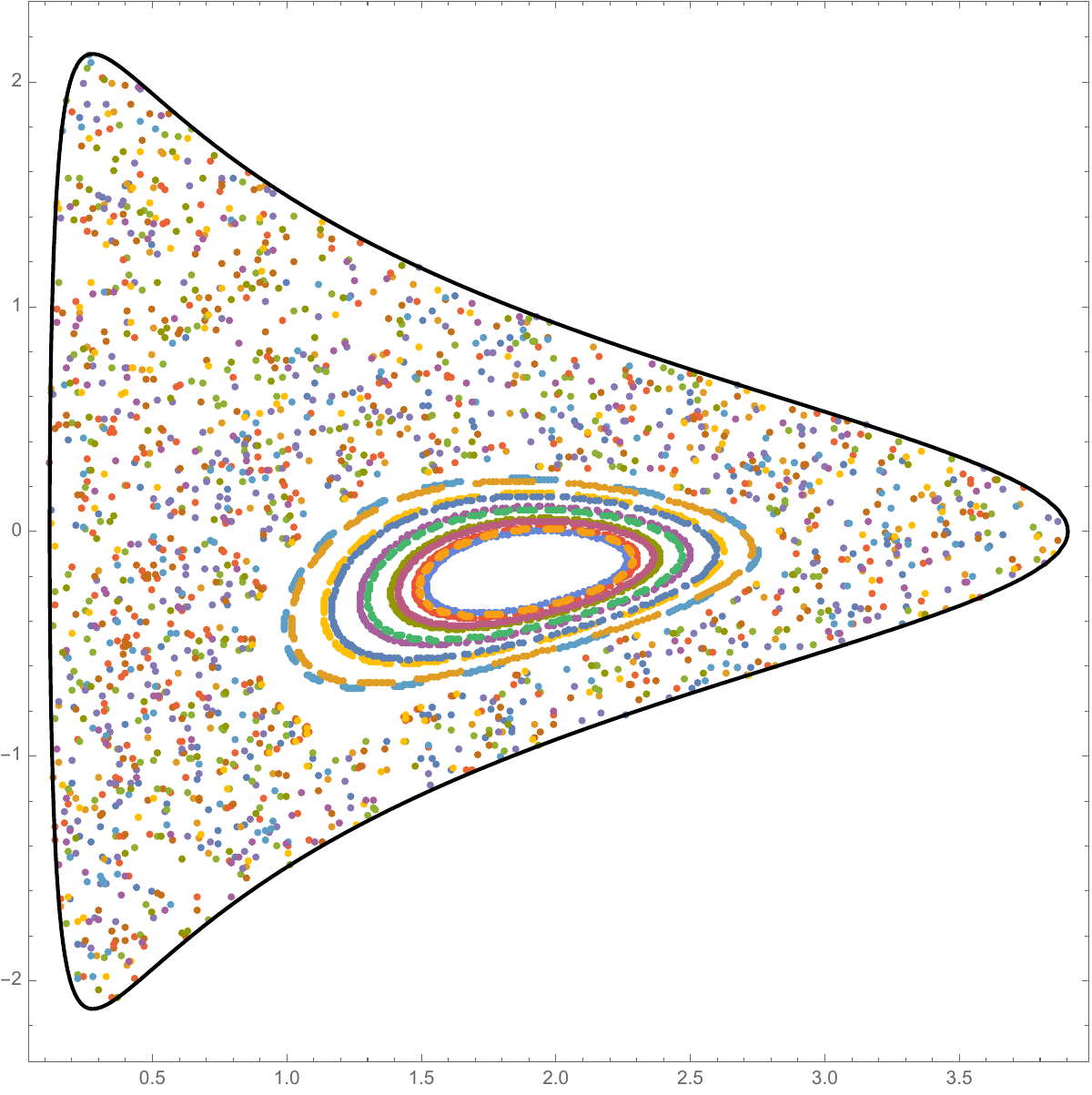}
\includegraphics[scale=0.18]{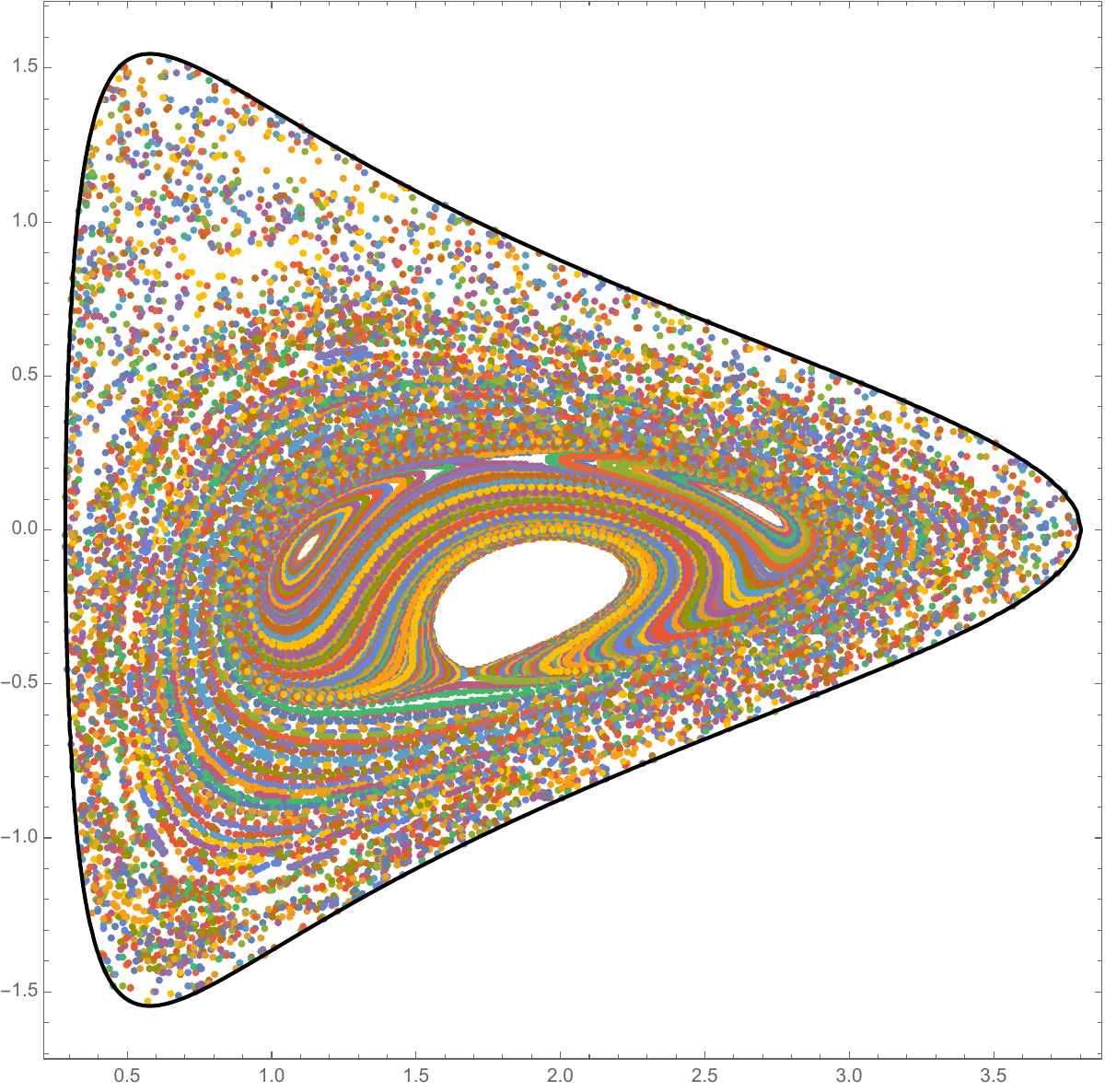}
\includegraphics[scale=0.18]{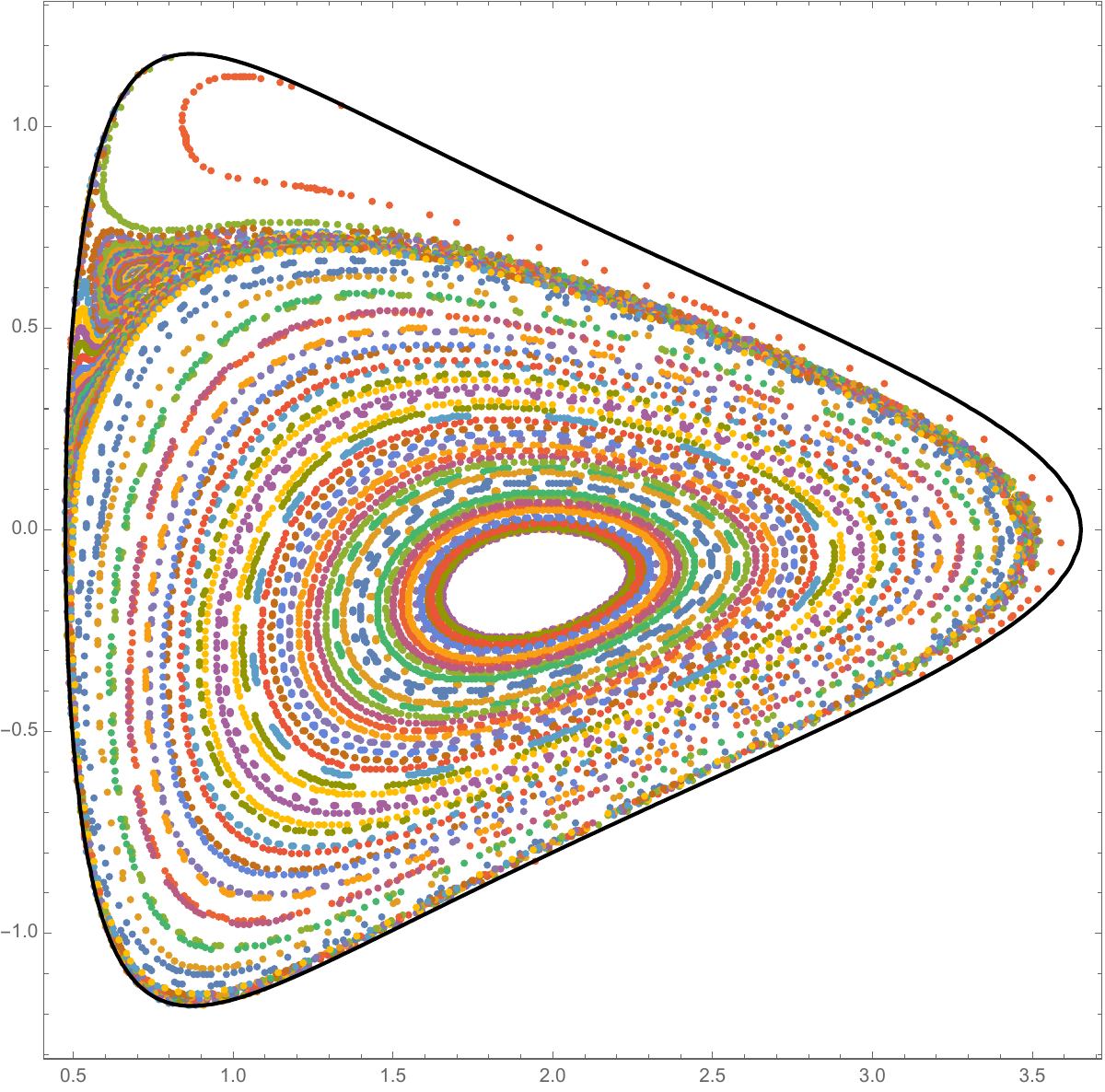}
\includegraphics[scale=0.18]{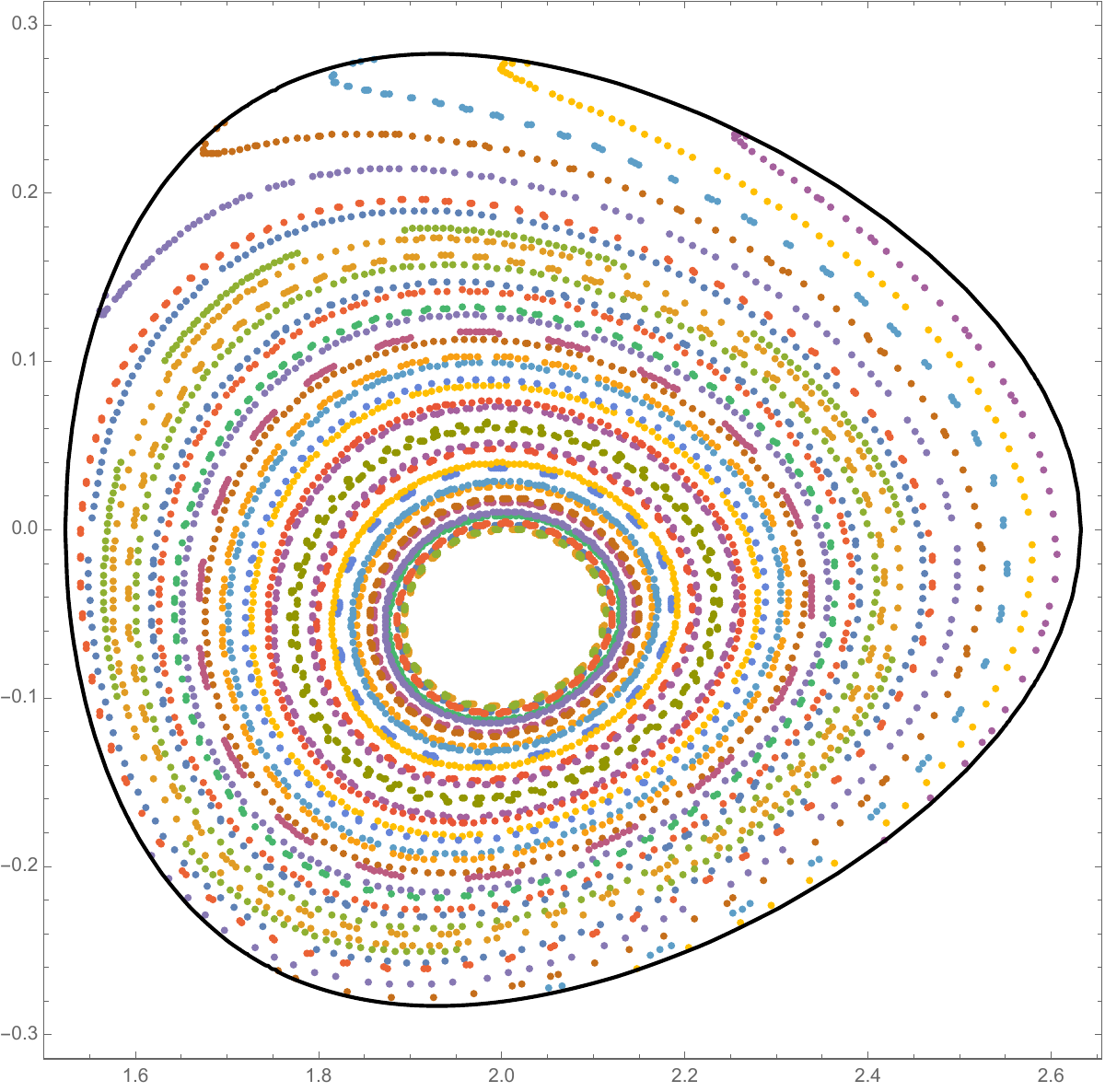}
\caption{\footnotesize Poincar\'e sections for $A \neq 0$: First and second rows with $A = 1/8,\,1/4$, respectively. In both cases, from left to right, $c = 0.35,\, 0.7,\,1,\, 1.8$. These sections are constructed using the unit scaling defined in \eqref{scaling}.}
\label{fig:SeccionesPoincareA14}
\end{figure}

\begin{figure}[h]
\includegraphics[scale=0.372]{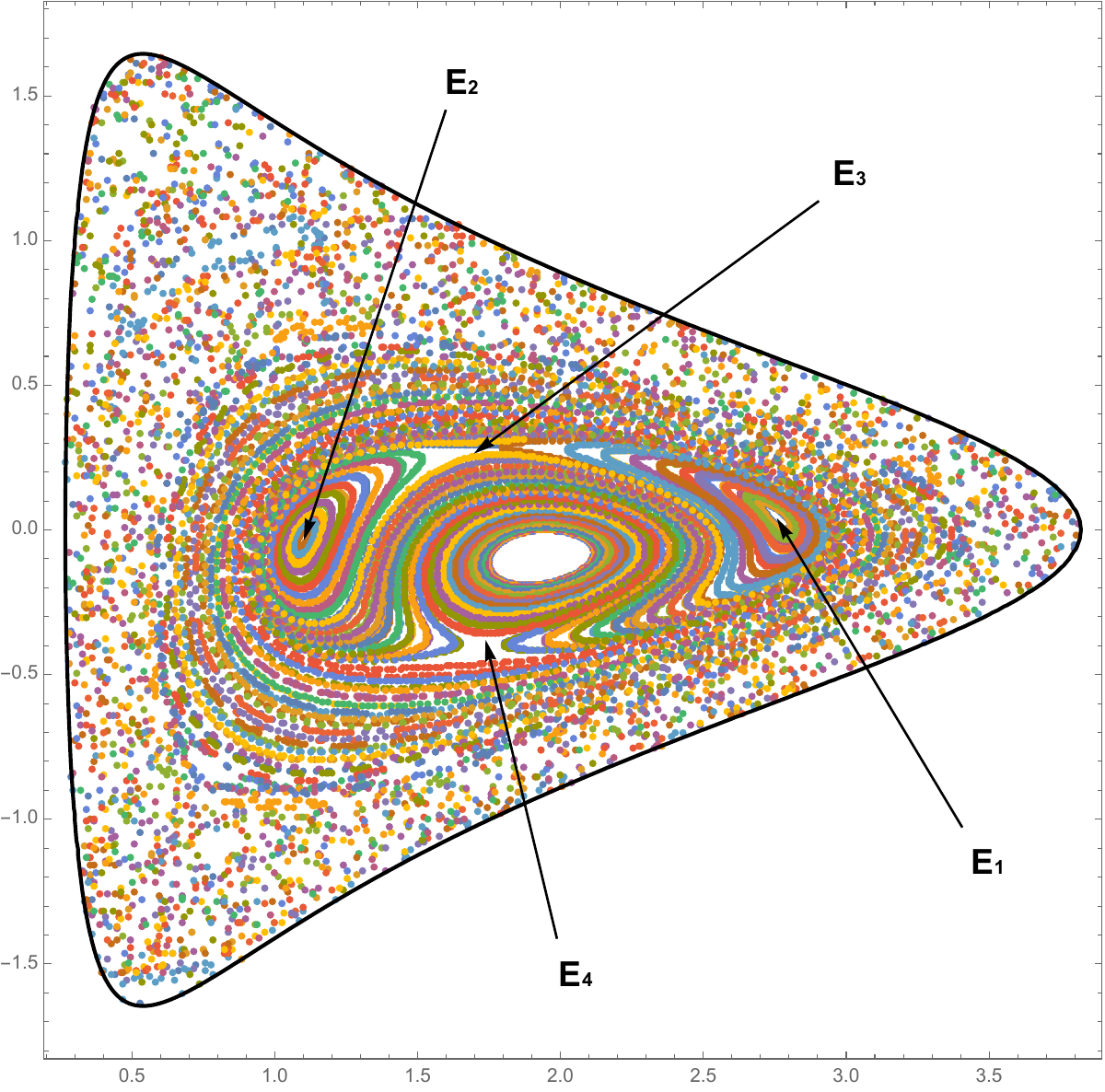}
\includegraphics[scale=0.372]{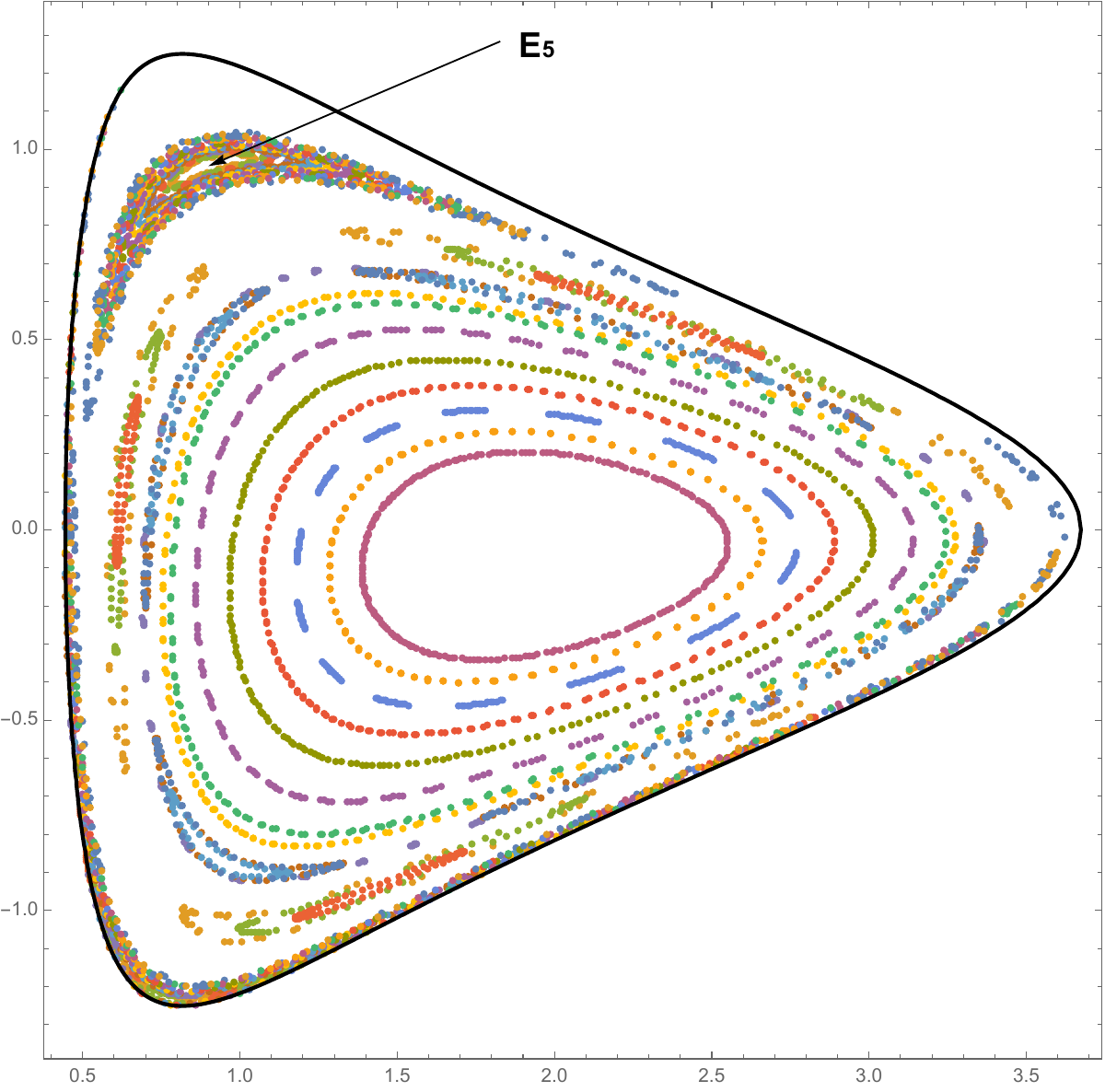}
\caption{\footnotesize The relative equilibria $E_i$ for $i=1,\ldots,5$ correspond with examples of quasi-periodic orbits for $A=1/8$. From left to right, we consider $c=0.7$ and $c=1$ respectively.}
\label{fig:SeccionesPoincareQuasiPO}
\end{figure}
Reduced-periodic orbits will be detected through suitable Poincar\'e sections, allowing us to analyze the influence of the parameter $A$ and the integral $P_\theta$ in the dynamics. Precisely, the case $A=0$ was studied in \cite{RiaguasElipeLara1999}, where the authors detected several quasi-periodic orbits. 

\begin{figure}[h!]
\includegraphics[scale=0.372]{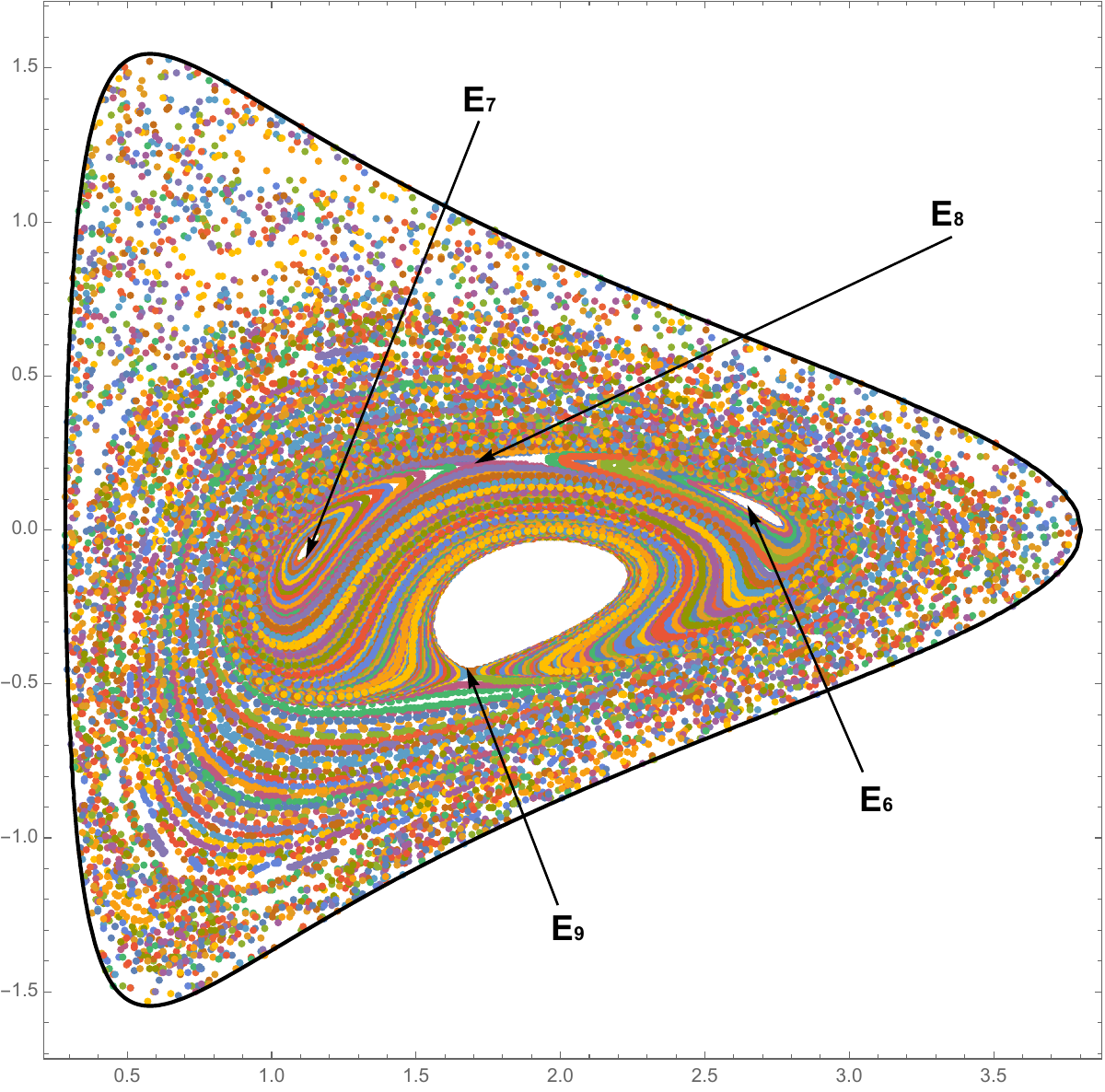}
\includegraphics[scale=0.372]{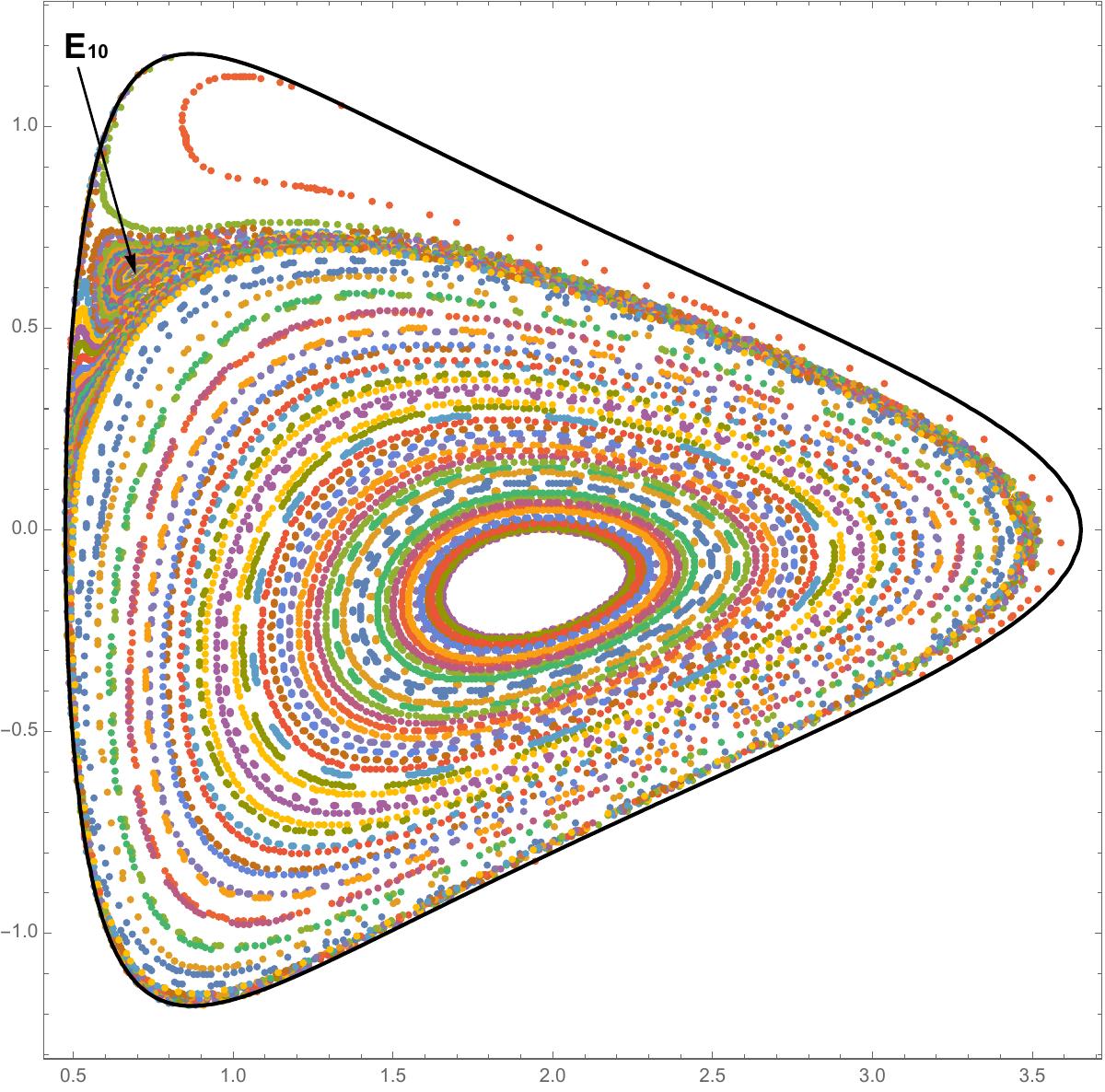}
\caption{\footnotesize The relative equilibria $E_i$ for $i=6,\ldots,10$ correspond with examples of quasi-periodic orbits for $A=1/4$.  From left to right, we consider $c=0.7$ and $c=1$ respectively.}
\label{fig:SeccionesPoincareQuasiPObis}
\end{figure}

Here, we compute Poincar\'e sections in the plane $(r,P_r)$ for $x=0$ and $P_x>0$. Figure~\ref{fig:SeccionesPoincareA0}, and Figure~\ref{fig:SeccionesPoincareA14} contain Poincar\'e sections for the cases $A=0$, $A\neq0$ respectively. Regardless of the value of $A$, these sections show a well-organized toroidal structure for high values of $P_\theta$. As the integral $P_\theta$ decreases, topological changes in the phase space are embodied through islands, elliptic and hyperbolic points, and an increasing chaotic area. We continue denoting $P_\theta=c$ as in the above section.

\begin{figure}[h]
\includegraphics[scale=0.337]{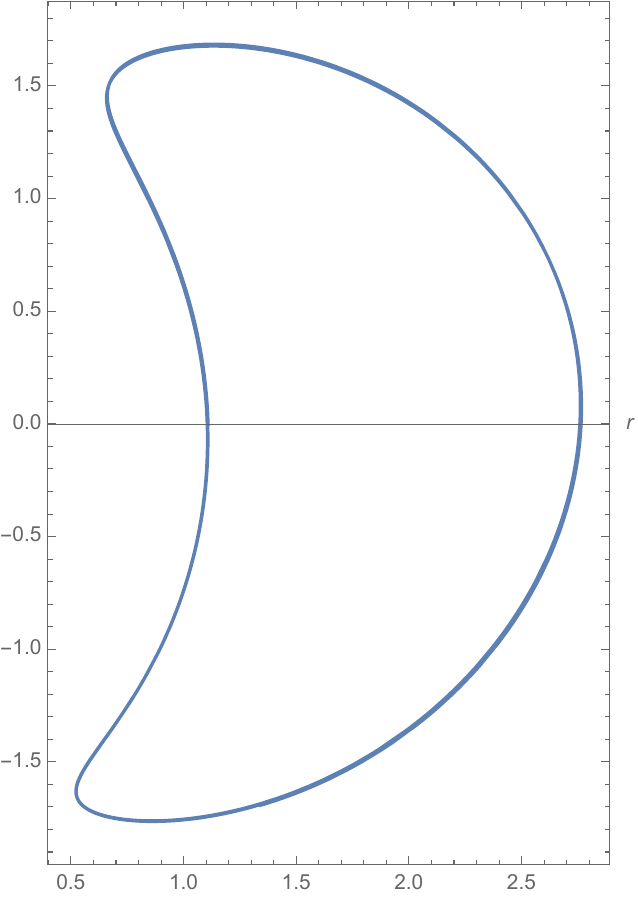}
\includegraphics[scale=0.337]{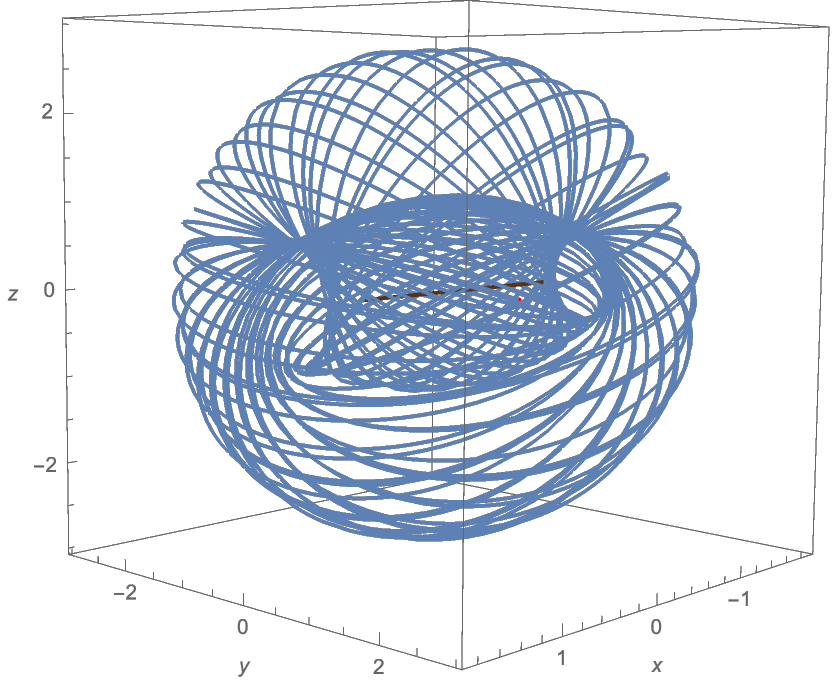}
\includegraphics[scale=0.337]{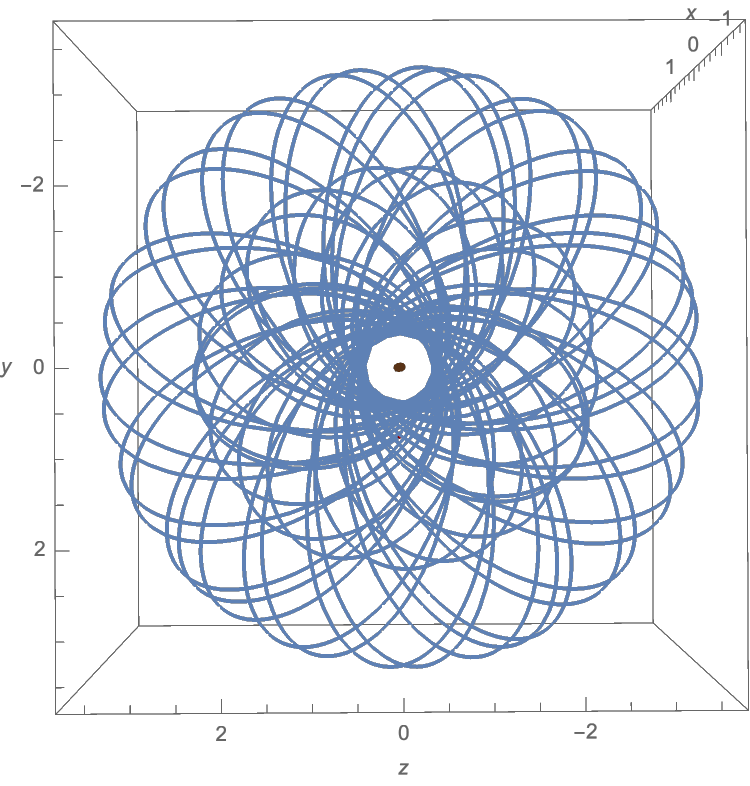}
\includegraphics[scale=0.337]{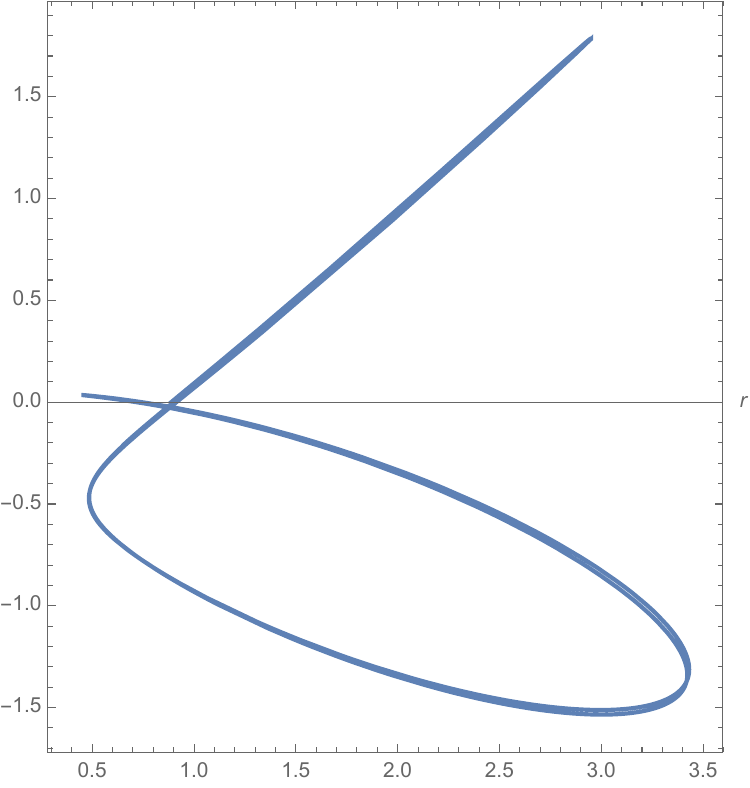}
\includegraphics[scale=0.337]{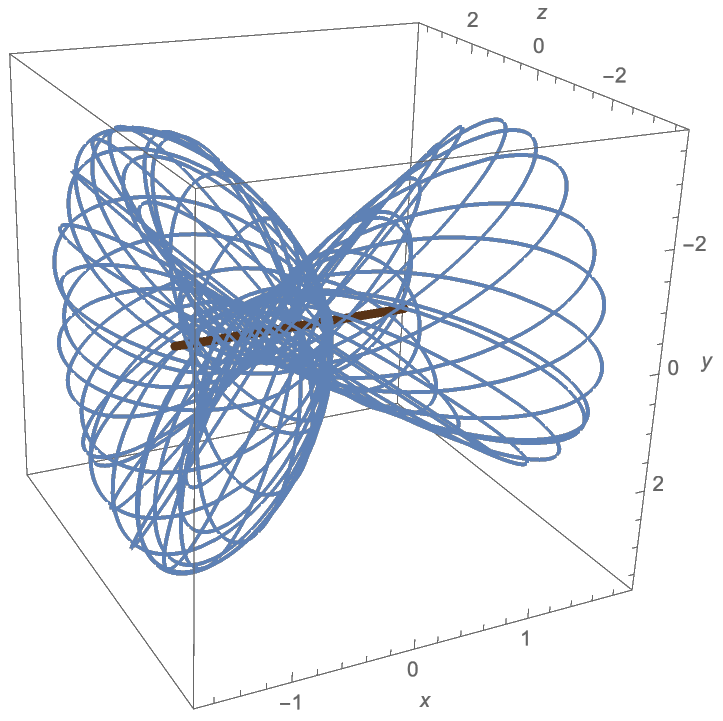}
\includegraphics[scale=0.337]{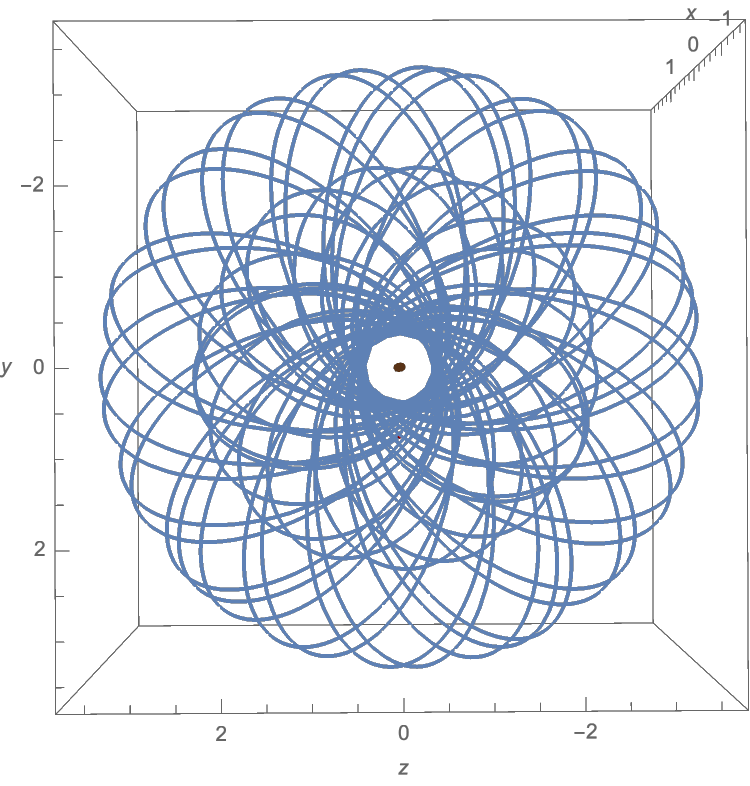}
\caption{\footnotesize The first row corresponds to the equilibria $E_1=(1.10845, 0,-0.0398045,1.34386)$ and the first figure displays $2D$ projection in the $rx$-plane, and the following figures are two different views of the $3D$ orbit in configuration space, illustrating the orbit over several revolutions of the angle $\theta$. The second row corresponds with the equilibria $E_5=(0.887805, 0,0.957145, 0.79564 )$ and shows figure $2D$ y $3D$.}
\label{fig:Simulations}
\end{figure}

This behavior was also observed in \cite{RiaguasElipeLara1999} for the constant density model. However, the case $A=0$ produces completely symmetric elliptic and hyperbolic points, while $A\neq0$ introduces asymmetries in the spatial distribution of these points. 
Notice that we deal with the case $A=0$ differently. Since this case was studied before in the literature  \cite{Riaguas1999, RiaguasElipeLara1999}, we facilitate the comparison with previous studies using the same length units. That is to say, for $A=0$, we consider $2L$ as the length unit, while for $A\neq0$, we follow the scaling of units given in \eqref{scaling}, which considers $L$ as the length unit.

\begin{figure}[h]
\includegraphics[scale=0.37]{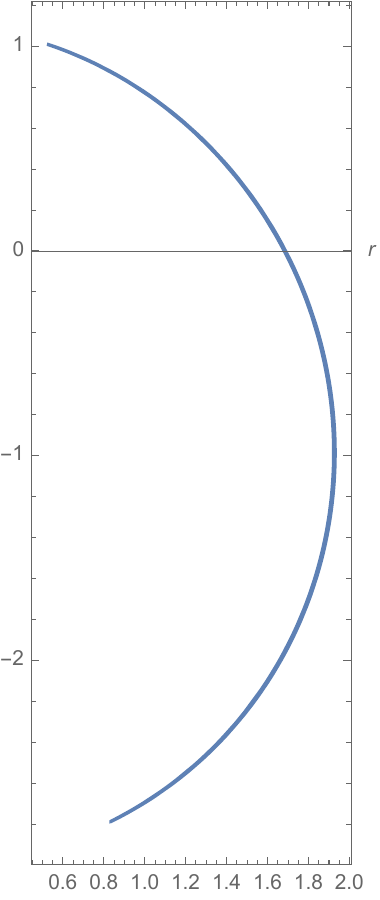}
\includegraphics[scale=0.35]{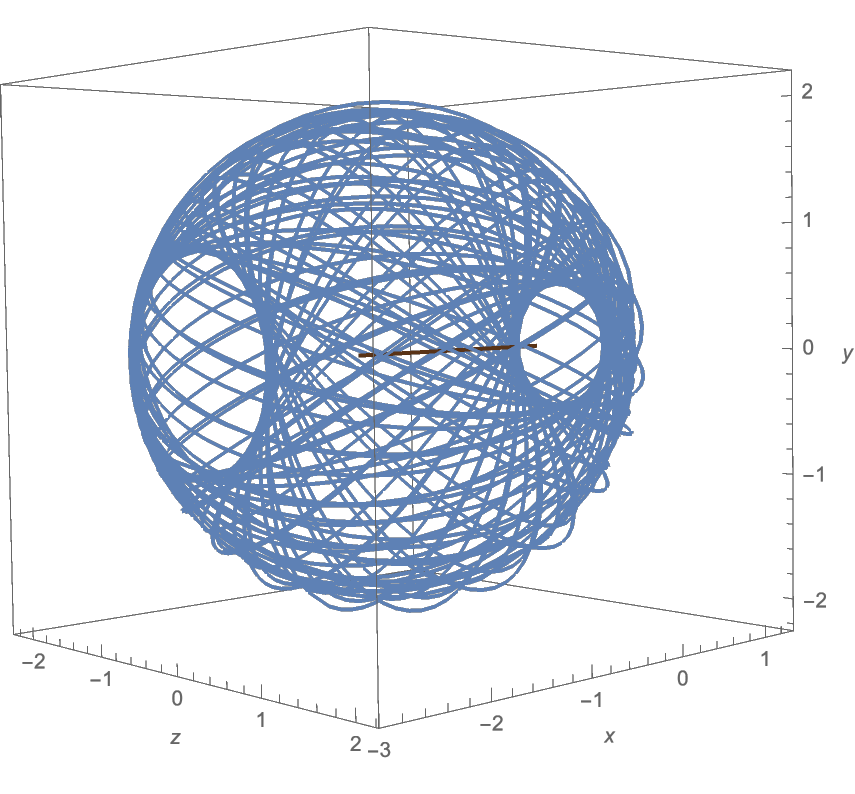}
\includegraphics[scale=0.39]{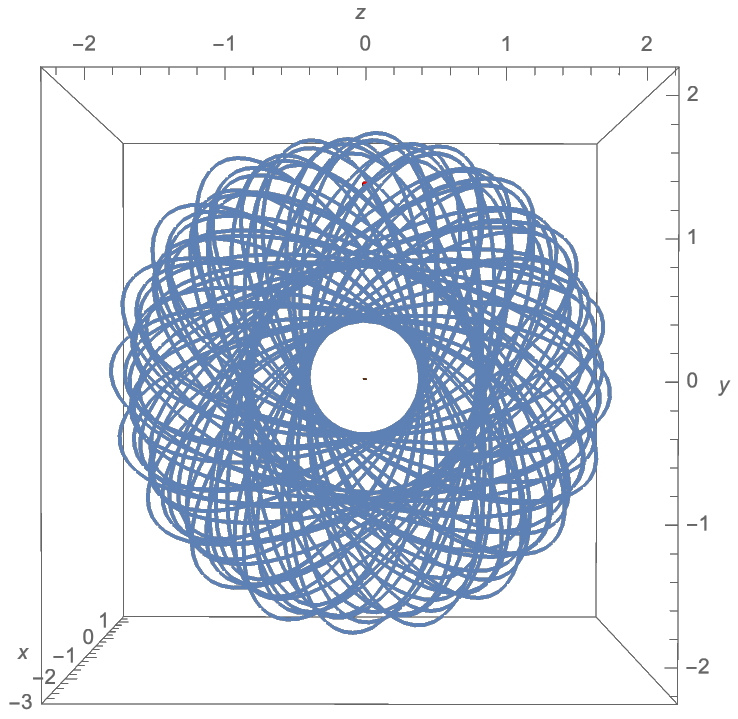}
\includegraphics[scale=0.37]{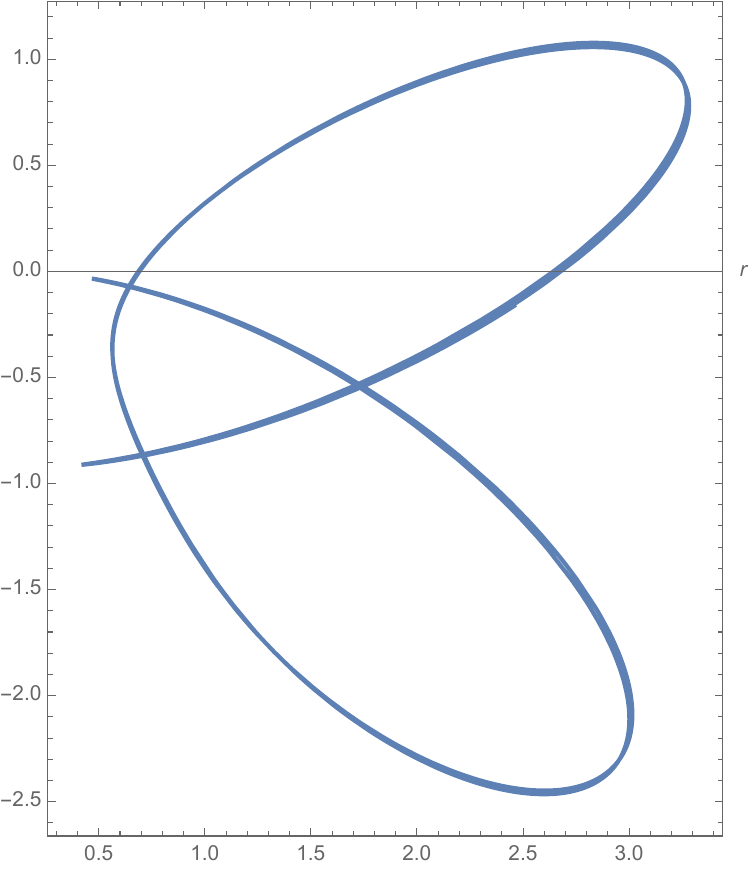}
\includegraphics[scale=0.35]{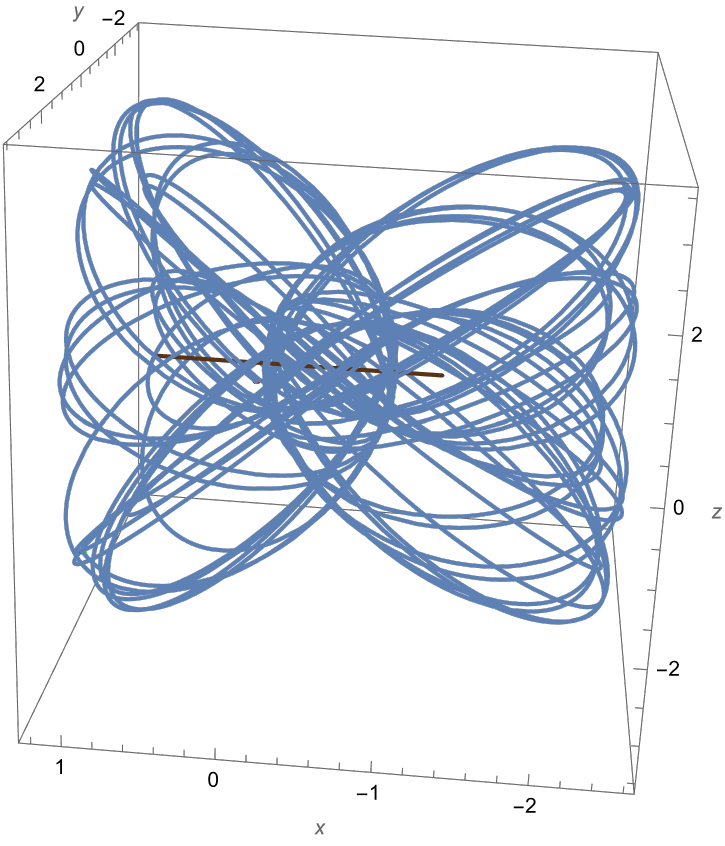}
\includegraphics[scale=0.39]{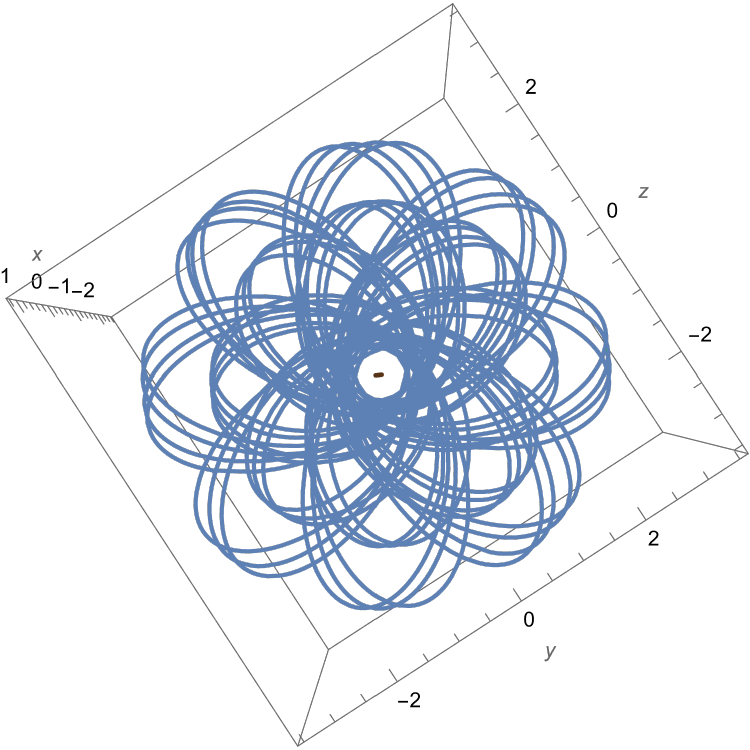}
\caption{\footnotesize The first row corresponds to the equilibria $E_9=(1.68132, 0, -0.46653, 0.900399)$. The first figure displays $2D$ projection in the $rx$-plane, and the following figures are two different views of the $3D$ orbit in configuration space, illustrating the orbit over several revolutions of the angle $\theta$. The second row corresponds with the equilibria $E_{10}=(0.690006, 0, 0.639448, 0.915063)$ and shows figure $2D$ y $3D$.
}
\label{fig:Simulations2}
\end{figure}

A detailed analysis of the elliptic and hyperbolic points showed in the Poincar\'e sections allows us to obtain periodic orbits in the reduced space $(r,x,P_r,P_x)$. Some examples of initial conditions leading to quasi-periodic orbits are given in Figure~\ref{fig:SeccionesPoincareQuasiPO}, where we choose $c=0.7$ and $c=1$ for the case $A=1/8$. Moreover, we provide some simulations of the mentioned quasi-periodic orbits in Figure~\ref{fig:Simulations} and Figure~\ref{fig:Simulations2}. Notice that in this simulations the particle remains in a close proximity of the segment, whose longitude is employed as length unit.

The relative equilibria tagged as $E_i$ for $i=1,\ldots,5$ in Figure~\ref{fig:SeccionesPoincareQuasiPO} and the relative equilibria tagged as $E_i$ for $i=6,\ldots,10$ in Figure~\ref{fig:SeccionesPoincareQuasiPObis} correspond to elliptic and hyperbolic points respectively in the Poincar\'e sections for $A=1/8$ and $A=1/4$ respectively. The accuracy on their initial conditions is limited by the detail provided in the Poincar\'e section.

\section*{Acknowledgement}
The author E.M. acknowledges support form the Universidad del B\'{i}o-B\'{i}o through a doctoral scholarship. The author J.V. was partially supported by ANID-Chile through FONDECYT Iniciaci\'on 11240582.

\bibliographystyle{siam}
\bibliography{Bibliography}

\end{document}